\documentclass{article}

\usepackage{amsmath,amsfonts,amssymb,amsthm}
\usepackage{graphicx,subfig}
\usepackage{tabularx}
\usepackage[utf8]{inputenc}
\usepackage{authblk}
\usepackage{xcolor}
\usepackage{verbatim}
\usepackage[title]{appendix}
\usepackage{paralist}
\usepackage{url}
\usepackage{mathtools}

\usepackage{graphicx}

\usepackage{xcolor}

\usepackage{lipsum}
\usepackage{amsfonts}
\usepackage{graphicx}
\usepackage{epstopdf}
\usepackage{algorithmic}
\usepackage{algorithm}

\usepackage{booktabs}
\usepackage{multirow}

\newcommand{\wal}[1]{\operatorname{wal}_{#1}}
\newcommand{\el}{\bm{\ell}}

\usepackage{bm}
\usepackage{amssymb}

\usepackage{hyperref}
\usepackage{comment}

\renewcommand{\le}{\leqslant}
\renewcommand{\ge}{\geqslant}
\renewcommand{\leq}{\leqslant}
\renewcommand{\geq}{\geqslant}
\renewcommand{\emptyset}{\varnothing}

\newcommand{\natu}{\mathbb{N}}

\newcommand{\bsk}{\boldsymbol{k}}

\newcommand{\bszero}{\boldsymbol{0}}
\newcommand{\bsone}{\boldsymbol{1}}

\newcommand{\bsalpha}{\boldsymbol{\alpha}}

\newcommand{\bsgamma}{\boldsymbol{\gamma}}

\newcommand{\med}{\mathrm{med}}

\newcommand{\trE}{\operatorname{tr}_E}

\newtheorem{theorem}{Theorem}[section]
\newtheorem{lemma}[theorem]{Lemma}
\newtheorem{proposition}[theorem]{Proposition}
\newtheorem{corollary}[theorem]{Corollary}
\newtheorem{Algorithm}[theorem]{Algorithm}

\theoremstyle{definition}

\theoremstyle{remark}
\newtheorem{remark}[theorem]{Remark}

\numberwithin{equation}{section}

\author[1]{Ziyang Ye}
\author[1]{Xiaoqun Wang}
\author[2]{Zexin Pan\thanks{Corresponding author}}

\date{}

\affil[1]{Department of Mathematical Sciences,
Tsinghua University,
Beijing 100084,
People's Republic of China}

\affil[2]{Institute of Fundamental and Transdisciplinary Research,
Zhejiang University,
866 Yuhangtang Road,
Xihu District,
Hangzhou,
Zhejiang Province 310058,
China}

\makeatletter
\let\AB@affilsepx\AB@affilsep

\makeatother

\title{Universal $L^2$-approximation using median digital-net algorithms}
\begin{document}

\maketitle

\begingroup
\renewcommand\thefootnote{}
\footnotetext{
Email addresses:
\url{yzy23@mails.tsinghua.edu.cn},
\url{wangxiaoqun@mail.tsinghua.edu.cn},
\url{zep002@zju.edu.cn}.

The work of the first and second authors was funded by the National Science Foundation of China
(Grant 72571153).
}
\endgroup

\begin{abstract}
We propose a median digital-net algorithm for $L^2$-approximation of non-periodic functions over $[0,1]^s$, inspired by the recently developed median lattice algorithms for the periodic setting. The algorithm requires no smoothness or weight parameters but only a sufficiently large candidate Walsh index set $K$. It proceeds in three stages: generating multiple estimates of the Walsh coefficients in $K$ using independent randomized digital-net samples; taking the respective median of both the estimates and their absolute values; then, based on these median values, identifying the dominant coefficients and  constructing a truncated Walsh series as the final approximation. We prove that if the target function has dominating mixed partial derivatives up to order $\alpha$, all having finite Vitali variation of fractional order $\lambda$, then the algorithm achieves an $L^2$-error of $\mathcal{O}(M^{-\alpha-\lambda+\eta})$ with high probability, where $M$ is the total number of function evaluations and $\eta>0$ is arbitrarily small. Furthermore, the implied constant grows at most polynomially in the dimension $s$ under suitable decay conditions on the ANOVA components of the target function. On the implementation side, we provide both parameter-dependent and -independent constructions of the index set $K$, and employ the fast Walsh--Hadamard transform and Gray code ordering to accelerate the algorithm.
Numerical experiments support the theoretical analysis and demonstrate that the proposed algorithm remains effective in high-dimensional settings.

\end{abstract}

\maketitle

\section{Introduction}\label{sec_1}
In this paper, we study the $L^2$-approximation of non-periodic functions on the unit cube $[0,1]^s$. This problem is of crucial importance in high-dimensional applications such as computational finance \cite{chen2021deep,chen2026deep}, machine learning \cite{bach2017breaking,adcock2024learning}, and uncertainty quantification \cite{narayan2021optimal,kaarnioja2022fast,gilbert2025multilevel}. However, the curse of dimensionality causes computational costs to grow exponentially with dimension, rendering traditional grid-based methods infeasible. This challenge motivates the search for novel function space assumptions and specialized algorithms.

To address this issue, quasi-Monte Carlo (QMC) methods offer a particularly promising direction. QMC methods are equal-weight quadrature rules that use low-discrepancy point sets, such as digital nets \cite{niederreiter1992random,dick2010digital} and lattices \cite{sloan1994lattice,dick2022lattice}, to tackle 
the curse of dimensionality. Although QMC methods are well established for numerical integration, adapting them to function approximation remains challenging. Existing results focus mainly on periodic functions, where the natural synergy between lattices and the Fourier series is exploited.

Pioneering work by Kuo et al. \cite{kuo2004lattice} established an $L^2$-approximation algorithm for functions in Korobov spaces (periodic Sobolev spaces). Their approach used a single rank-1 lattice of $M$ nodes, achieving an error rate of $\mathcal{O}(M^{-\alpha/2+\epsilon})$ for arbitrarily small $\epsilon>0$, where $\alpha$ is the smoothness parameter of the Korobov space. This rate is nearly optimal among algorithms restricted to function evaluations on a single rank-1 lattice \cite{byrenheid2017tight}, but falls short of the near‑optimal rate $\mathcal{O}(M^{-\alpha+\epsilon})$ without this restriction. More recent approaches, including multiple rank-1 lattices \cite{kammerer2019approximation}, subsampled lattices \cite{bartel2025minimal}, multiple shifted lattices with least-squares recovery \cite{cai2025lattice,du2026deterministic}, and median lattices \cite{pan2025l_2,pan2025universal}, all match this near-optimal rate with high probability.

We focus particularly on the median lattice algorithms introduced by Pan et al. \cite{pan2025universal}. Building on earlier work \cite{pan2025l_2}, their approach uses the median‑of‑means to estimate many Fourier coefficients from only a few randomized lattices. Its key advancement is universality: the algorithm requires no prior knowledge of the smoothness and weight parameters of the Korobov space. By leveraging the universal property of median QMC (see, e.g., \cite{goda2022construction,pan2025automatic,pan2026dimension}) and choosing a parameter‑independent index set of Fourier coefficients, the method achieves the near‑optimal error rate without parameter tuning.

To extend lattice-based methods to non-periodic settings, a common strategy is to employ the tent transformation to map functions into periodic spaces \cite{suryanarayana2016reconstruction,cools2016tent,kuo2021function}. Recent work has used this transformation to establish optimal error bounds for parameter-dependent approximation in half-period cosine spaces \cite{du2026deterministic}. In contrast, a direct digital-net approach, which does not require transforming the function, remains largely unexplored. To address this, we draw inspiration from the median QMC framework of \cite{pan2026dimension,pan2025universal} and develop a universal median digital-net algorithm for the $L^2$-approximation of non-periodic functions.

The proposed approach builds upon the Walsh series representation of non-periodic functions, analogous to the Fourier series in the periodic case.  Specifically, we consider target functions that admit an absolutely convergent Walsh series expansion
\begin{equation*}
   f(\bm{x})=\sum_{\bm{k}\in \mathbb{N}_0^s}\widehat{f}(\bm{k})\wal{\bm{k}}(\bm{x}),
\end{equation*}
where $\wal{\bm{k}}$ denotes the multivariate Walsh function and the Walsh coefficients are
\begin{equation*}
    \widehat{f}(\bm{k})=\int_{[0,1]^s}f(\bm{x})\wal{\bm{k}}(\bm{x})\,\mathrm{d}\bm{x}.
\end{equation*}

To recover the dominant Walsh coefficients of $f$, we employ digital nets, exploiting the structural correspondence between digital nets and Walsh functions. The procedure follows a three-stage design paralleling the one in \cite{pan2025universal}. First, for each coefficient in a preselected candidate index set $K$ (chosen to capture the dominant Walsh coefficients), we compute $R$ independent estimates using random digital-net samples of size $N=2^m$. We then apply a median-of-means estimator to both the coefficient estimates and their absolute values. The former yields the final coefficient estimates, while the latter serves as a selection criterion to identify the $N$ coefficients with the largest absolute values.
Finally, we construct a truncated Walsh series as an approximation to $f$ using the median estimates.

Following \cite{pan2026dimension}, we assume $f$ belongs to a function space equipped with a norm based on the fractional Vitali variation introduced in \cite{dick2008walsh}. However, our algorithm uses unweighted index sets instead of the weighted index set construction in \cite{pan2026dimension}. This choice slightly degrades the tractability, but it endows the algorithm with universality: it eliminates the dependence on specific weight parameters while maintaining strong convergence guarantees.

We prove that if $f$ possesses dominating mixed derivatives up to order $\alpha$, each having finite fractional Vitali variation of order $\lambda$ (formally defined in Section~\ref{subsec_ANOVA}), the algorithm achieves an $L^2$-error of $\mathcal{O}(N^{-\alpha-\lambda+\eta})$ for any $\eta>0$ with high probability. This convergence rate holds whenever the candidate index set $K$ contains the dominant Walsh coefficients determined by $\alpha$ and $\lambda$, and $R$ grows logarithmically in the size of $K$. In particular, for the universal index set $K^{\operatorname{univ}}(0)$ constructed in Section~\ref{subsec:univ}, which covers all admissible $\alpha$ and $\lambda$, the probabilistic error bound is $\mathcal{O}(s^{\alpha+\lambda}M^{-\alpha-\lambda+\eta})$ under the tractability condition~\eqref{condition_gamma}, where $M=NR$ is the total number of function evaluations and the implied constant is independent of the dimension $s$. 

To enable efficient implementation, we employ the fast Walsh–Hadamard transform and Gray code ordering to significantly reduce the cost of estimating Walsh coefficients over all indices in $K$, thereby making the universal approximation approach practical for large-scale problems.

The remainder of this paper is structured as follows. Section \ref{sec_2} reviews the necessary background, including digital nets, Walsh functions, ANOVA decomposition, and the function norm used in \cite{pan2026dimension}. Section \ref{sec_3} presents the universal median digital-net algorithm, constructs the ideal candidate index set under full knowledge of the parameters, and states the probabilistic $L^2$-error bound (Theorem \ref{thm_main}) along with the tractability result. Section \ref{sec_4} focuses on the practical implementation, providing two concrete constructions of the index set $K$ and discussing acceleration strategies via the fast Walsh--Hadamard transform and Gray code ordering. Section \ref{sec_5} reports numerical experiments that validate the theoretical findings. Section \ref{sec_6} offers a concluding discussion. The appendix contains the omitted proofs.

\section{Notation and background}\label{sec_2}
Throughout the paper, $f:[0,1]^s\to \mathbb{R}$ denotes the target function to be approximated,  where $s$ represents the dimension of its domain. We assume that $f$ is nonzero.
Let $\mathbb{N}$ denote the natural numbers and $\mathbb{N}_0=\mathbb{N}\cup \{0\}$. In particular, we define $\mathbb{N}_*^s=\mathbb{N}_0^s\setminus\{\bm{0}\}$.
Let $\mathbb{Z}_{\le \ell}=\{0,1,\ldots,\ell\}$ for $\ell\in\mathbb{N}_0$ and $\mathbb{Z}_{<\ell}=\{0,1,\ldots,\ell-1\}$ for $\ell\in\mathbb{N}$. We further denote $1{:}s=\{1,\ldots,s\}$. For a vector $\bm{x}\in\mathbb{R}^s$ and a subset $u\subseteq 1{:}s$, $\bm{x}_u$ denotes the subvector of $\bm{x}$ indexed by $u$, while $\bm{x}_{u^c}$ denotes the subvector indexed by $1{:}s\setminus u$. The indicator function $\bm{1}\{A\}$ equals $1$ if event
$A$ occurs and $0$ otherwise. The notation $\mathbb{U}(S)$ denotes the uniform distribution over a set $S$.
The cardinality of a set $K$ is $|K|$. Furthermore, for a finite $K\subseteq\mathbb{N}_0$, we define $\Vert K\Vert=\sum_{\ell\in K}\ell$. For vector $\bm{k}=(k_1,\ldots,k_s)\in\mathbb{N}_0^s$, its support is denoted by
\begin{equation*}
    \bm{s}(\bm{k})=\{j\in 1:s|k_j\ne 0\}.
\end{equation*}

We assume the target function $f$ is continuous, denoted as $f\in C([0,1]^s)$, with $\Vert f\Vert_{\infty}=\max_{\bm{x}\in[0,1]^s}|f(\bm{x})|$. The $L^p$-norm of $f$ is $\Vert f\Vert_{L^p}=\left(\int_{[0,1]^s}|f(\bm{x})|^p\,\mathrm{d}\bm{x}\right)^{1/p}$, and $L^p([0,1]^s)$ denotes the space of functions with finite $L^p$-norm.

To denote the classical partial derivatives of $f$, we adopt the following notation
\begin{equation*}
    f^{(\bm{\alpha})}(\bm{x})=\left(\prod_{j=1}^{s}\frac{\partial^{\alpha_j}}{\partial x_j^{\alpha_j}}\right)f(\bm{x}),\quad\forall \bm{\alpha}=(\alpha_1,\dots,\alpha_s)\in\mathbb{N}_0^s.
\end{equation*}

For integrand $g$ defined on $[0,1]^s$, QMC methods approximate the integral
\begin{equation*}
    \mu(g):=\int_{[0,1]^s}g(\bm{x})\,\mathrm{d}\bm{x}\quad\text{by}\quad\widehat{\mu}(g)=\frac{1}{N}\sum_{i=0}^{N-1}g(\bm{x}_i)
\end{equation*}
with $\bm{x}_0,\ldots,\bm{x}_{N-1}\in [0,1]^s$. In this paper, we will choose $\bm{x}_i$  to be the randomized base-2 digital nets described below, and we denote the number of points in the digital net by $N=2^m$.

\subsection{Digital net and complete random design}
For $m\in\mathbb{N}_0$ and $i\in\mathbb{Z}_{<2^m}$, we express $i$ in its binary form $i=\sum_{\ell=1}^{m}i_{\ell}2^{\ell-1}$ and denote $\Vec{i}=(i_1,\ldots,i_m)^\top\in \{0,1\}^m$. Analogously, for $a\in [0,1)$, we approximate its infinite binary expansion $a = \sum_{\ell=1}^{\infty}a_{\ell}2^{-\ell}$ by truncating to $E\in\mathbb{N}$ digits, yielding $\Vec{a}=(a_1,\ldots,a_E)^\top\in\{0,1\}^E$. Here, we adopt the convention that binary expansions do not contain an infinite tail of $1$’s, ensuring that the representation is unique. 

Furthermore, for any $k,t\in\mathbb{N}_0$ with binary representations $k=\sum_{\ell=1}^{\infty}k_{\ell}2^{\ell-1}$ and $t = \sum_{\ell=1}^{\infty}t_{\ell}2^{\ell-1}$, we define the \textbf{bitwise XOR} operator $\oplus$ as $w = k\oplus t$, where the binary representation $w = \sum_{\ell=1}^{\infty}w_{\ell}2^{\ell-1}$ satisfies $w_\ell = k_\ell + t_\ell \mod{2}$ for all $\ell$. For the $s$-dimensional case, the operator $\oplus$ is applied component-wise. 

Given $m\in \natu_0$ and $E\in \natu$, a \textbf{base-$2$ digital net} consists of $2^m$ points in $[0,1]^s$ determined by $s$ generating matrices $C_j\in \{0,1\}^{E\times m}$. The deterministic points $\bm{x}_i=(x_{i1},\ldots,x_{is})$ are constructed via
\begin{equation*}
    \Vec{x}_{ij}=C_j\Vec{i}\mod{2}\quad\text{for }i\in\mathbb{Z}_{<2^m}, j\in 1{:}s,
\end{equation*}
where $\Vec{x}_{ij}$  represents the first $E$ bits of $x_{ij}$. Typically, for unrandomized digital nets we have $E\le m$.

To incorporate randomness, we set the precision $E$ such that $E\ge m$ and modify the construction as follows
\begin{equation}\label{digital_net_C_D}
    \Vec{x}_{ij}=C_j\Vec{i} + \Vec{D}_j\mod{2},
\end{equation}
where $C_j\in\{0,1\}^{E\times m}$ is a random matrix and $\Vec{D}_j\in\{0,1\}^{E}$ is a random vector. The vector $\Vec{D}_j$ is called the \textbf{random digital shift} and consists of independent Bernoulli (0.5) entries. Regarding the randomization of $C_j$, we employ the \textbf{complete random design} \cite{pan2025automatic}, in which every entry of $C_j$ is sampled independently from $\mathbb{U}(\{0,1\})$. 

Now, let $\{\bm{x}_{i}^E,i\in \mathbb{Z}_{<2^m}\}$ denote the $E$-bit precision points generated by equation (\ref{digital_net_C_D}) with matrices $\bm{C}_{m,E}:=\{C_1,\ldots,C_s\}\subseteq \{0,1\}^{E\times m}$ and shifts $\bm{D}_{E}:=\{\Vec{D}_1,\ldots,\Vec{D}_s\}\subseteq \{0,1\}^{E}$. We denote the QMC estimator for an integral $\mu(g)$ as
\begin{equation}\label{digt_QMC}
    \widehat{\mu}_{\bm{C}_{m,E},\bm{D}_{E}}(g):=\frac{1}{2^m}\sum_{i=0}^{2^m-1}g(\bm{x}_i^E). 
\end{equation}
In this paper, our theoretical analysis is primarily based on the idealized case where $E=\infty$. In practical algorithms, $E$ is taken to be sufficiently large. The main theorem is nevertheless established for finite $E$, with the proof combining the $E=\infty$ analysis and stability arguments for the finite-precision setting.

\subsection{Walsh functions and error analysis for base-2 digital
nets}
Walsh functions are the natural Fourier basis for the analysis of base-2 digital nets. For $k\in\mathbb{N}_0$ and $x\in [0,1)$, the $k$'th univariate Walsh function $\wal{k}$ is defined by 
\begin{equation*}
    \wal{k}(x)=(-1)^{\Vec{k}^\top \Vec{x}},
\end{equation*}
where $\Vec{k}\in\{0,1\}^{\infty}$ and $\Vec{x}\in\{0,1\}^{\infty}$ are the binary expansions of $k$ and $x$, respectively. Since $\Vec{k}$ has only finitely many nonzero digits, the inner product $\Vec{k}^\top \Vec{x}$ is well-defined. It follows that if $k<2^E$, then $\wal{k}(x)$ is constant on elementary intervals of the form  $[t/2^E,(t+1)/2^E)$ for $t\in\mathbb{Z}_{<2^E}$, taking values in $\{-1,1\}$.

For the $s$-dimensional case, the multivariate Walsh function $\wal{\bm{k}}:[0,1)^s\to \{-1,1\}$ is defined as the tensor product of univariate Walsh functions
\begin{equation*}
    \wal{\bm{k}}(\bm{x})=\prod_{j=1}^{s}\wal{k_j}(x_j)=(-1)^{\sum_{j=1}^{s}\Vec{k}^\top_j\Vec{x}_j},
\end{equation*}
where $\bm{k}=(k_1,\ldots,k_s)\in\mathbb{N}_0^s$. These functions satisfy
\begin{equation*}
    \wal{\bm{k}\oplus\el}(\bm{x})=\wal{\bm{k}}(\bm{x})\wal{\el}(\bm{x}),\quad\forall \bm{k},\el\in\mathbb{N}_0^s.
\end{equation*}

The functions $\{\wal{\bm{k}}\}_{\bm{k}\in\mathbb{N}_0^s}$ constitute a complete orthonormal basis for $L^2([0,1]^s)$ \cite{dick2010digital}. Therefore, any $g\in L^2([0,1]^s)$ can be expressed via the Walsh expansion
\begin{equation}\label{wal_expan}
    g(\bm{x})=\sum_{\bm{k}\in \mathbb{N}_0^s}\widehat{g}(\bm{k})\wal{\bm{k}}(\bm{x})
\end{equation}
with Walsh coefficients given by
\begin{equation*}
    \widehat{g}(\bm{k})=\int_{[0,1]^s}g(\bm{x})\wal{\bm{k}}(\bm{x})\,\mathrm{d}\bm{x},
\end{equation*}
where the series in \eqref{wal_expan} converges in the $L^2$-sense.  Using this expansion, we
obtain the following error decomposition for the QMC estimator.
\begin{lemma}\label{lemma_error_g}
    Suppose $f\in C([0,1]^s)$ satisfies $\sum_{\el\in \mathbb{N}_0^s}|\widehat{f}(\el)|<\infty$. For any $\bm{k}\in\mathbb{N}_0^s$, let $g(\bm{x})=f(\bm{x})\wal{\bm{k}}(\bm{x})$. Then $\mu(g)=\widehat{f}(\bm{k})$  and
    \begin{equation}\label{mu_g_error}
        \widehat{\mu}_{\bm{C}_{m,\infty},\bm{D}_{\infty}}(g)-\mu(g)=\sum_{\el\in\mathbb{N}_*^s}Z(\el)W(\el)\widehat{f}(\bm{k}\oplus\el),
    \end{equation}
    where $ \widehat{\mu}_{\bm{C}_{m,\infty},\bm{D}_{\infty}}(g)$ is given by equation~\eqref{digt_QMC} with $E=\infty$, and
    \begin{align}
        Z(\el)=\bm{1}\left\{\sum_{j=1}^{s}\Vec{\ell}_j^\top C_j = \bm{0}\mod{2}\right\}\quad\text{and} \quad W(\el)=(-1)^{\sum_{j=1}^{s}\Vec{\ell}_j^\top \Vec{D}_j}.\nonumber
    \end{align}
\end{lemma}
\begin{proof}
In the $L^2$-sense, we have 
\begin{align}
    g(\bm{x})&=f(\bm{x})\wal{\bm{k}}(\bm{x})=\sum_{\el\in\mathbb{N}_0^s}\widehat{f}(\el)\wal{\el}(\bm{x})\wal{\bm{k}}(\bm{x})\nonumber\\
    &=\sum_{\el\in\mathbb{N}_0^s}\widehat{f}(\el)\wal{\el\oplus\bm{k}}(\bm{x})=\sum_{\el\in\mathbb{N}_0^s}\widehat{f}(\el\oplus \bm{k})\wal{\el}(\bm{x}).\label{g_wal_exp}
\end{align}
This implies $\mu(g)=\widehat{g}(\bm{0})=\widehat{f}(\bm{k})$. Since $f$ is continuous with $\sum_{\el\in \mathbb{N}_0^s}|\widehat{f}(\el)|<\infty$, \cite[Theorem A.20]{dick2010digital} shows that 
\begin{equation*}
    f(\bm{x})=\sum_{\bm{k}\in \mathbb{N}_0^s}\widehat{f}(\bm{k})\wal{\bm{k}}(\bm{x})
\end{equation*}
holds pointwise. Consequently, equality (\ref{g_wal_exp}) holds pointwise, and equation \eqref{mu_g_error} follows from the same argument as in \cite[Theorem 1]{pan2025automatic}.
\end{proof}

\subsection{ANOVA decomposition and function norm}\label{subsec_ANOVA}
In this subsection, we focus on the function norm considered in \cite[Section 3]{pan2026dimension}. First, we employ the following ANOVA decomposition \cite{liu2006estimating} for $f\in C([0,1]^s)$
\begin{equation*}
    f = \sum_{u\subseteq 1{:}s}f_u,
\end{equation*}
where each component $f_u$ depends only on $\bm{x}_u$ and satisfies $\int_{0}^{1}f_u(\bm{x})\,\mathrm{d}x_j=0$ for any $j\in u$. This decomposition is unique in the $L^2$-sense.
Furthermore, if we enforce $f_u\in C([0,1]^s)$, then $f_u$ is given by the following recursive formula
\begin{equation}\label{f_u}
    f_u(\bm{x})=\int_{[0,1]^{s-|u|}}f(\bm{x})\,\mathrm{d}\bm{x}_{u^c}-\sum_{v\subsetneq u} f_{v}(\bm{x}),
\end{equation}
with $f_{\emptyset}(\bm{x})=\mu(f)$. By an abuse of notation, we write $f_u(\bm{x})=f_{u}(\bm{x}_u)$, where the latter $f_u:[0,1]^{|u|}\to \mathbb{R}$ is a $|u|$-dimensional function.

The ANOVA component $f_u$ has the Walsh expansion
\begin{equation}\label{wal_f_u}
f_u(\bm{x})=\sum_{\bm{s}(\bm{k})=u}\widehat{f}(\bm{k})\wal{\bm{k}}(\bm{x}),
\end{equation}
since $\wal{\bm{k}}$ depends only on $\bm{x}_{\bm{s}(\bm{k})}$ and $\int_{0}^{1}\wal{\bm{k}}(\bm{x})\,\mathrm{d}x_j=0$ whenever $j\in\bm{s}(\bm{k})$. 

Next, we introduce the fractional Vitali variation with $\mathfrak{p}=2$ from \cite{dick2008walsh} to quantify the variation of the ANOVA components $f_u$. For a function $g:[0,1]^d\to\mathbb{R}$ and a subcube $J=\prod_{j=1}^s[a_j,b_j)\subseteq[0,1)^d$, the alternating sum over $J$ is given by 
\begin{equation*}
    \Delta(g,J)=\sum_{v\subseteq 1{:}d}(-1)^{d-|v|}g(\bm{a}_v,\bm{b}_{v^c}),
\end{equation*}
where $(\bm{a}_v,\bm{b}_{v^c})$ denotes the vertex of $J$ with $x_j=a_j$ for $j\in v$ and $x_j=b_j$ for $j\in 1{:}d\setminus v$. For $\lambda\in (0,1]$, the fractional Vitali variation of order $\lambda$ is defined as
\begin{equation*}
    V_{\lambda}^{(d)}(g)=\sup_{\mathcal{P}}\left(\sum_{J\in\mathcal{P}}\operatorname{Vol}(J)\left|\frac{\Delta(g,J)}{\operatorname{Vol}(J)^{\lambda}}\right|^2\right)^{1/2},
\end{equation*}
where the supremum is over all partitions $\mathcal{P}$ of $[0,1)^d$ into subcubes and $\operatorname{Vol}(J)$ is the volume of subcube $J$. When $d=0$, we set $V_\lambda^{(0)}(g)=|g|$ by convention. 

Furthermore, we recall the partial fractional Vitali variation introduced in \cite{pan2025automatic}. For $v\subseteq 1{:}d$, define the partial integral operator $I_v:L^1([0,1]^d)\to L^1([0,1]^{|v|})$ as 
\begin{equation*}
    I_v(g)(\bm{x}_v)=\int_{[0,1]^{d-|v|}}g(\bm{x}_v,\bm{x}_{v^c})\,\mathrm{d}\bm{x}_{v^c}.
\end{equation*}
Let $D_{v^c}$ be the class of functions
\begin{equation*}
    D_{v^c}:=\left\{\rho_{v^c}: \rho_{v^c}(\bm{x})=\prod_{j\in v^c} \rho_j(x_j),\rho_j\in C([0,1]), \rho_j\ge 0, \Vert\rho_j\Vert_{\infty}\le 1 \right\}.
\end{equation*}
The partial fractional Vitali variation of order $\lambda$ with respect to $v\subseteq 1:d$ is given by 
\begin{equation*}
    V_\lambda^v(g):=\sup_{\rho_{v^c}\in D_{v^c}}V_{\lambda}^{(|v|)}(I_v(g\rho_{v^c})).
\end{equation*}
By convention, $V_\lambda^{1{:}d}(g)=V_\lambda^{(d)}(g)$. An important special case arises when $\lambda=1$, where we have the estimate 
\begin{equation}\label{eqn:V1vfL2}
 V^{v}_1(g)\leq \Vert g^{(\bsone_{v},\bszero_{v^c})}\Vert_{L^2}   
\end{equation}
if the weak derivative $g^{(\bsone_{v},\bszero_{v^c})}$ exists and belongs to $L^2([0,1]^d)$.
Here $(\bsone_{v},\bszero_{v^c})$ denotes the vector whose $j$-th entry equals $1$ if $j\in v$ and equals $0$ otherwise.

We now define the function norm for the target function $f$. For $\alpha\in\mathbb{N}_0,\lambda\in (0,1]$ and a nonempty subset $u\subseteq 1{:}s$, define
\begin{equation}\label{f_u_a_l}
    \Vert f\Vert_{u,\alpha,\lambda} = \sup_{v\subseteq u}\sup_{\substack{\bm{\alpha}_u\in\mathbb{Z}_{\le\alpha}^{|u|}\\ \alpha_j=\alpha\  \forall j\in v \\ \alpha_j>0\ \forall j\in u\setminus v}  } V_\lambda^v(f_u^{(\bm{\alpha_u})}),
\end{equation}
where $f_u$ is the ANOVA component of $f$ given by (\ref{f_u}). When $\lambda=1$ and the weak mixed derivative $f^{(\alpha+1,\dots,\alpha+1)}$ exists and belongs to $L^2([0,1]^s)$, equation~\eqref{eqn:V1vfL2} implies the bound
$$\Vert f\Vert_{u,\alpha,1}\leq \sup_{\bsalpha_u\in \{1:(\alpha+1)\}^{|u|}} \Vert f_u^{(\bsalpha_u)}\Vert_{L^2}.$$

We further set $\Vert f\Vert_{\emptyset,\alpha,\lambda} = |\mu(f)|$ and define the global variation norm
\begin{equation*}
    \Vert f\Vert_{s,\alpha,\lambda}=\sup_{u\subseteq 1{:}s}\Vert f\Vert_{u,\alpha,\lambda}.
\end{equation*}
For nonzero $f$, $\Vert f\Vert_{s,\alpha,\lambda}>0$, and we define the \textbf{relative variation} of ANOVA component $f_u$ as 
\begin{equation}\label{gamma_u}
    \gamma_u=\frac{\Vert f\Vert_{u,\alpha,\lambda}}{\Vert f\Vert_{s,\alpha,\lambda}},
\end{equation}
which satisfies $\sup_{u\subseteq 1{:}s}\gamma_{u}=1$. We will sometimes refer to $\gamma_u$ as weights because they serve as weight parameters in the construction of anisotropic index sets.

\section{The median algorithm and its error}\label{sec_3}
In this section, we introduce the median digital-net algorithm and construct the ideal Walsh index set used in our analysis. We then present the main result, Theorem \ref{thm_main}, which provides a probabilistic bound on the $L^2$-approximation error. Its proof is deferred to Appendix \ref{appendix_main}. We conclude this section with a discussion on the tractability properties of our median digital-net algorithm.

\subsection{The universal median digital-net algorithm}
To illustrate the core ideas of the proposed method, we first present Algorithm \ref{alg:idealized_median}, an idealized algorithm inspired by \cite[Algorithm 3.1]{pan2025universal} that uses complete random designs with precision $E=\infty$ and assumes no rounding errors.

\begin{Algorithm}\label{alg:idealized_median}
For given $m\in\mathbb{N}$, odd $R\in\mathbb{N}$, finite subset $K\subseteq \mathbb{N}_0^s$ and target function $f:[0,1]^s\to\mathbb{R}$, perform the following steps:
\begin{enumerate}
    \item \textbf{For} $r = 1$ to $R$, do:
    \begin{enumerate}
        \item Randomly draw $C_{1,r},\ldots,C_{s,r}$ i.i.d. uniform from $\{0,1\}^{\infty \times m}$.
        \item Randomly draw $\vec{D}_{1,r},\ldots,\vec{D}_{s,r}$ i.i.d. uniform from $\{0,1\}^{\infty}$.
        \item Denote $N = 2^m$. For $i$ from $0$ to $N-1$, compute $\mathbf{x}_{i,r}=(x_{i1,r},\dots,x_{is,r})$ by
        \[
        \vec{x}_{ij,r}=C_{j,r}\vec{i} + \vec{D}_{j,r}\pmod{2}.
        \]
        \item For $i$ from $0$ to $N-1$, evaluate $f$ at $\mathbf{x}_{i,r}$ and store $f(\mathbf{x}_{i,r})$.
        \item For ${\bsk}\in K$, compute
        \begin{equation}\label{coe_in_alg1}
               \widehat{f}_{N,r}({\bsk})=\frac{1}{N}\sum_{i=0}^{N-1}f(\mathbf{x}_{i,r})\operatorname{wal}_{{\bsk}}(\mathbf{x}_{i,r}).
        \end{equation}
    \end{enumerate}
    \item Order ${\bsk}\in K$ such that the medians of $|\widehat{f}_{N,r}({\bsk})|$ over $r\in 1{:}R$ are in descending order, i.e.,
    \[
    \operatorname*{median}_{r\in 1{:}R}|\widehat{f}_{N,r}({\bsk}_1)|
    \ge
    \operatorname*{median}_{r\in 1{:}R}|\widehat{f}_{N,r}({\bsk}_2)|
    \ge \cdots
    \]
    and set $H_N:=\{{\bsk}_1,\ldots,{\bsk}_N\}$ if $|K|\ge N$, else $H_N=K$.
    \item Define the median estimator
    \[
    \widehat{f}_{N,\mathrm{med}}({\bsk}):=\operatorname*{median}_{r\in 1{:}R}\widehat{f}_{N,r}({\bsk})
    \]
    and the corresponding approximation
    \[
    A_{m}(f)(\mathbf{x}):=\sum_{{\bsk}\in H_N}\widehat{f}_{N,\mathrm{med}}({\bsk})\operatorname{wal}_{{\bsk}}(\mathbf{x}).
    \]
\end{enumerate}
\end{Algorithm}

In practice, our algorithm operates with finite precision $E$ and is subject to rounding errors, which are assumed to be bounded by the machine epsilon $\epsilon>0$. Under these constraints, we obtain the practical variant summarized in Algorithm \ref{alg:ep_median}.

\begin{Algorithm}\label{alg:ep_median}
For given $m\in\mathbb{N}$, odd $R\in\mathbb{N}$, finite subset $K\subseteq \mathbb{N}_0^s$, target function $f:[0,1]^s\to\mathbb{R}$, precision $E\in\mathbb{N}$ and machine epsilon $\epsilon>0$, perform the following steps:
\begin{enumerate}
    \item \textbf{For} $r = 1$ to $R$, do:
    \begin{enumerate}
        \item Randomly draw $C^E_{1,r},\ldots,C^E_{s,r}$ i.i.d. uniform from $\{0,1\}^{E\times m}$.
        \item Randomly draw $\vec{D}^E_{1,r},\ldots,\vec{D}^E_{s,r}$ i.i.d. uniform from $\{0,1\}^{E}$.
        \item Denote $N=2^m$. For $i$ from $0$ to $N-1$, compute $\mathbf{x}_{i,r}^E=(x^E_{i1,r},\dots,x^E_{is,r})$ by
        \begin{equation}\label{E_digit}
        \vec{x}^E_{ij,r}=C^E_{j,r}\vec{i} + \vec{D}^E_{j,r}\pmod{2}.
        \end{equation}
        \item For $i$ from $0$ to $N-1$, evaluate $f$ at $\mathbf{x}_{i,r}^E$ and store $f(\mathbf{x}_{i,r}^E)$.
        \item For ${\bsk}\in K$, compute $\widehat{f}^{E,\epsilon}_{N,r}({\bsk})$ subject to rounding errors:
        \begin{equation}\label{coe_in_alg2}
        \widehat{f}^{E,\epsilon}_{N,r}({\bsk})=\frac{1}{N}\sum_{i=0}^{N-1}f(\mathbf{x}_{i,r}^E)\operatorname{wal}_{{\bsk}}(\mathbf{x}_{i,r}^E) + \epsilon_{\bsk,r} \quad \text{for}\quad |\epsilon_{\bsk,r}|\leq \epsilon.
        \end{equation}
    \end{enumerate}
    \item Order ${\bsk}\in K$ such that the medians of $|\widehat{f}^{E,\epsilon}_{N,r}({\bsk})|$ over $r\in 1{:}R$ are in descending order, i.e.,
    \[
    \operatorname*{median}_{r\in 1{:}R}|\widehat{f}^{E,\epsilon}_{N,r}({\bsk}_1)|
    \ge
    \operatorname*{median}_{r\in 1{:}R}|\widehat{f}^{E,\epsilon}_{N,r}({\bsk}_2)|
    \ge \cdots
    \]
    and set $H^{E,\epsilon}_N:=\{{\bsk}_1,\ldots,{\bsk}_N\}$ if $|K|\ge N$, else $H^{E,\epsilon}_N=K$.
    \item For ${\bsk}\in H^{E,\epsilon}_N$, define the median estimator
    \[
    \widehat{f}^{E,\epsilon}_{N,\mathrm{med}}({\bsk}):=\operatorname*{median}_{r\in 1{:}R}\widehat{f}^{E,\epsilon}_{N,r}({\bsk})
    \]
    and return the median digital-net $L^2$-approximation for $f$
    \[
    A^{E,\epsilon}_{m}(f)(\mathbf{x}):=\sum_{{\bsk}\in H^{E,\epsilon}_N}\widehat{f}^{E,\epsilon}_{N,\mathrm{med}}({\bsk})\operatorname{wal}_{{\bsk}}(\mathbf{x}).
    \]
\end{enumerate}
\end{Algorithm}

The following theorem characterizes the discrepancy between the two algorithms in estimating Walsh coefficients $\widehat{f}(\bm{k})$.

\begin{theorem}\label{thm_dis}
    Assume that for each $1\le r\le R$ and $1\le j\le s$, the matrix $C^E_{j,r}$ in Algorithm \ref{alg:ep_median} is given by the first $E$ rows of the matrix $C_{j,r}$ in Algorithm \ref{alg:idealized_median}, and likewise, the vector $\Vec{D}^E_{j,r}$ in Algorithm \ref{alg:ep_median} comprises the first $E$ entries of the vector $\Vec{D}_{j,r}$ in Algorithm \ref{alg:idealized_median}. If the precision $E$ satisfies
    \begin{equation}\label{condition_E}
        2^E>\max_{\bm{k}=(k_1,\ldots,k_s)\in K}\max_{1\le j\le s }k_j,
    \end{equation}
    then for each $\bm{k}\in K$ we have 
    \begin{equation}\label{assumption_ep}
        |\widehat{f}_{N,r}(\bm{k})-\widehat{f}^{E,\epsilon}_{N,r}(\bm{k})|\le \epsilon+\omega_f(\sqrt{s}2^{-E}),
    \end{equation}
    where $\widehat{f}_{N,r}(\bm{k})$ and $\widehat{f}^{E,\epsilon}_{N,r}(\bm{k})$ are estimators of $\widehat{f}(\bm{k})$ defined in (\ref{coe_in_alg1}) and (\ref{coe_in_alg2}), respectively, and 
    \begin{equation*}
        \omega_f(\delta)=\sup\{|f(\bm{x})-f(\bm{x}')|:\bm{x},\bm{x}'\in[0,1]^s,\Vert \bm{x}-\bm{x}'\Vert_2\le \delta\}\quad\text{for }\delta>0
    \end{equation*}
    is the modulus of continuity of $f$.
\end{theorem}

\begin{proof}
    The construction of the digital net points $\bm{x}_{i,r}$ in Algorithm \ref{alg:idealized_median} and $\bm{x}_{i,r}^E$ in Algorithm \ref{alg:ep_median} ensures that each entry $x_{ij,r}^E$ of $\bm{x}_{i,r}^E$ has an $E$-digit binary expansion, which coincides with the first $E$ digits of the binary expansion of the corresponding entry $x_{ij,r}$ in $\bm{x}_{i,r}$. Consequently, both $x_{ij,r}$ and $x_{ij,r}^E$ lie in the same interval $[t/2^E,(t+1)/2^E)$ for some $0\le t\le 2^E-1$.
    
    Now for any $\bm{k}\in K$, since $2^E>k_j$, we have $\wal{k_j}(x_{ij,r})=\wal{k_j}(x_{ij,r}^E)$. So $\wal{\bm{k}}(\bm{x}_{i,r})=\wal{\bm{k}}(\bm{x}_{i,r}^E)\in\{-1,1\}$. Hence,
    \begin{equation*}
        |f(\bm{x}_{i,r})\wal{\bm{k}}(\bm{x}_{i,r})-f(\bm{x}_{i,r}^E)\wal{\bm{k}}(\bm{x}_{i,r}^E)|=|f(\bm{x}_{i,r})-f(\bm{x}_{i,r}^E)|.
    \end{equation*}
    Using \cite[Lemma 1]{pan2024super}, it follows that
    \begin{align*}
        &|\widehat{f}_{N,r}(\bm{k})-\widehat{f}^{E,\epsilon}_{N,r}(\bm{k})|\\
        &\le |\epsilon_{\bm{k},r}|+\frac{1}{N}\sum_{i=0}^{N-1}|f(\bm{x_{i,r}})\wal{\bm{k}}(\bm{x_{i,r}})-f(\bm{x}_{i,r}^E)\wal{\bm{k}}(\bm{x}_{i,r}^E)|\\
        &=| \epsilon_{\bm{k},r}|+\frac{1}{N}\sum_{i=0}^{N-1}|f(\bm{x_{i,r}})-f(\bm{x}_{i,r}^E)|\\
        &\le \epsilon + \omega_f(\sqrt{s}2^{-E}).\qedhere
    \end{align*}
\end{proof}

Throughout the remainder of the paper, we write
\begin{equation}\label{constant_ep1}
    \epsilon_1=\epsilon + \omega_f(\sqrt{s}2^{-E}).
\end{equation}
We note that $\epsilon_1$ is negligible for sufficiently small $\epsilon$ and sufficiently large $E$.

\subsection{Anisotropic index sets for theoretical analysis}
In this subsection, we define anisotropic index sets needed for theoretical analysis, followed by a discussion of their relevant properties.

For $k\in\mathbb{N}_0$, we write the binary representation of $k$ differently by 
\begin{equation*}
k=\sum_{j=1}^{v}2^{c_j-1},    
\end{equation*}
where $v\in\mathbb{N}_0$ and $c_1>c_2>\cdots>c_v\ge 1$. For $\alpha\in\mathbb{N}_0$ and $\lambda\in (0,1]$, define
\begin{equation*}
\mu_{\alpha,\lambda}(k)=\begin{cases}
    \sum_{j=1}^v c_j, &\text{if } v\le \alpha,\\
    \lambda c_{\alpha+1} + \sum_{j=1}^{\alpha}c_j, &\text{if } v>\alpha.
\end{cases}
\end{equation*}
The weight $\mu_{\alpha,\lambda}$ serves as a fractional extension of the weight $\mu_\alpha$ introduced by Dick  \cite{dick2008walsh,dick2009decay}.
For $T\geq 0$ and $u\subseteq 1{:}s$, we consider the index set
\begin{equation}\label{K_u}
    K_{u,\alpha,\lambda}(T):=\left\{\bm{k}\in \mathbb{N}_0^s:\bm{s}(\bm{k})=u,\mu_{\alpha,\lambda}(\bsk)\le T \right\},
\end{equation}
where $\mu_{\alpha,\lambda}(\bsk)=\sum_{j\in \bm{s}(\bm{k})}\mu_{\alpha,\lambda}(k_j)$ if $\bsk\neq \bszero$ and $\mu_{\alpha,\lambda}(\bszero)=0$. By convention, we let $K_{u,\alpha,\lambda}(T) = \emptyset$ if $T < 0$. This definition of $K_{u,\alpha,\lambda}(T)$ is consistent with the one given in \cite[Section 4]{pan2026dimension}, although it is expressed here more compactly in terms of $\mu_{\alpha,\lambda}$. Consequently, an upper bound on the cardinality $|K_{u,\alpha,\lambda}(T)|$ is provided by \cite[Lemma 6]{pan2026dimension}.

\begin{lemma}\label{lemma_K_u(T)}
    For $\alpha\in\mathbb{N}_0,\lambda\in(0,1]$ and $u\subseteq 1{:}s$,
    \begin{equation*}
        |K_{u,\alpha,\lambda}(T)|\le 2^{T/(\alpha+\lambda)}\left(A_{\alpha,\lambda}\max\left\{T/(\alpha+\lambda),1\right\} + B_{\alpha,\lambda} \right)^{|u|},
    \end{equation*}
    where
    \begin{equation}\label{constant_A_B}
        A_{\alpha,\lambda} = \frac{1}{\alpha!(2^{1/(\alpha+\lambda)}-1)^\alpha}\quad\text{and}\quad
        B_{\alpha,\lambda} = \sum_{\alpha'=1}^{\alpha}\frac{1}{(\alpha')! (2^{1/(\alpha+\lambda)}-1)^{\alpha'}}.
    \end{equation}
\end{lemma}

Now for $\gamma_u$ defined in (\ref{gamma_u}), denote
\begin{equation}\label{Gamma(m)}
    \Gamma_{\alpha,\lambda,\beta,\bm{\gamma}}(m):=1+\sum_{u\subseteq 1:s}\gamma_u^{\frac{1}{\alpha+\lambda+1/2}} \beta^{-|u|} (A_{\alpha,\lambda}m + B_{\alpha,\lambda})^{|u|},
\end{equation}
where $A_{\alpha,\lambda},B_{\alpha,\lambda}$ are given in (\ref{constant_A_B}), and $\beta \ge 1$ is a shrinkage parameter reserved for later analysis. For $\delta\in (0,1)$, let
\begin{align}\label{T_u_m_gamma}
    T^{}_{u,m,\alpha,\lambda,\beta,\bm{\gamma},\delta}:=(\alpha+\lambda)\bigg[&m+\log_2\delta-\log_2\left(\Gamma_{\alpha,\lambda,\beta,\bm{\gamma}}(m)\right)\\ 
    & +\frac{\log_2(\gamma_u)}{\alpha+\lambda+1/2} -|u|\log_2\beta\bigg].\nonumber
\end{align}
We denote the anisotropic index set
\begin{equation}\label{K_gamma}
    K_{m,\alpha,\lambda,\beta,\bm{\gamma},\delta}:=\bigcup_{u\subseteq 1:s}K_{u,\alpha,\lambda}(T_{u,m,\alpha,\lambda,\beta,\bm{\gamma},\delta}).
\end{equation}
The following lemma bounds the size of the index set.

\begin{lemma}\label{card}
  The cardinality of $K_{m,\alpha,\lambda,\beta,\bm{\gamma},\delta}$ satisfies  $|K_{m,\alpha,\lambda,\beta,\bm{\gamma},\delta}|\le \delta 2^m$.
\end{lemma}
\begin{proof}
    According to Lemma \ref{lemma_K_u(T)} and $T_{u,m,\alpha,\lambda,\beta,\bm{\gamma},\delta}\le m$, we have 
    \begin{align*}
        |K_{m,\alpha,\lambda,\beta,\bm{\gamma},\delta}|&=\sum_{u\subseteq 1:s}K_{u,\alpha,\lambda}(T_{u,m,\alpha,\lambda,\beta,\bm{\gamma},\delta})\\
        &\le \sum_{u\subseteq 1:s}2^{T_{u,m,\alpha,\lambda,\beta,\bm{\gamma},\delta}/(\alpha+\lambda)}(A_{\alpha,\lambda}m + B_{\alpha,\lambda})^{|u|}\\
        &=\frac{\delta 2^m}{\Gamma_{\alpha,\lambda,\beta,\bm{\gamma}}(m)}\sum_{u\subseteq 1:s}\gamma_u^{\frac{1}{\alpha+\lambda+1/2}}\beta^{-|u|}(A_{\alpha,\lambda}m + B_{\alpha,\lambda})^{|u|}\\
        &\le \delta 2^m.\qedhere
    \end{align*}
\end{proof}

The next two lemmas provide upper bounds for the $\el^2$ and $\el^1$ sums of Walsh coefficients outside the anisotropic index set $K_{m,\alpha,\lambda,\beta,\bm{\gamma},\delta}$, respectively. As the proofs are primarily technical, we relegate them to the Appendix.

\begin{lemma}\label{lemma_lambda}
    For $\theta \in (0,1)$ and $f$ with $\Vert f\Vert_{s,\alpha,\lambda}<\infty$, we have
    \begin{equation}\label{Lambda_m}
     \Lambda_2:=\sum_{\bm{k}\in\mathbb{N}_0^s\setminus K_{m,\alpha,\lambda,\beta,\bm{\gamma},\delta}}|\widehat{f}(\bm{k})|^2\le \Upsilon_{\alpha,\lambda,\beta, \bm{\gamma},\delta,\theta}\frac{\Gamma_{\alpha,\lambda,\beta,\bm{\gamma}}(m)^{2\theta(\alpha+\lambda)}}{2^{2\theta(\alpha+\lambda)m}} \Vert f\Vert_{s,\alpha,\lambda}^2,
    \end{equation}
    where $\Gamma_{\alpha,\lambda,\beta,\bm{\gamma}}(m)$ is defined in (\ref{Gamma(m)}), and
    \begin{equation}\label{constant_Gamma}
    \Upsilon_{\alpha,\lambda,\beta, \bm{\gamma},\delta,\theta}=\delta^{-2\theta(\alpha+\lambda)}\sum_{u\subseteq 1:s}\gamma_u^{2-\frac{2\theta(\alpha+\lambda)}{\alpha+\lambda+1/2}}\beta^{2\theta(\alpha+\lambda)|u|} C_{\alpha,\lambda,\theta}^{|u|},
    \end{equation}
    with 
    \begin{equation}\label{constant_C}
        C_{\alpha,\lambda,\theta}=\frac{4^{-\alpha}}{\alpha!(4^{(\alpha+\lambda)(1-\theta)}-1)(4^{1-\theta}-1)^{\alpha}}+\sum_{\alpha'=0}^{\alpha}\frac{4^{-\alpha'+1}}{(\alpha')!(4^{1-\theta}-1)^{\alpha'}}.
    \end{equation}
\end{lemma}
\begin{proof}
    See Appendix \ref{appendix_lambda}.
\end{proof}

\begin{lemma}\label{lemma_psi}
Assume $\alpha+\lambda>1/2$.
    For $\eta\in (0,1-1/(2\alpha+2\lambda))$ and $f$ with $\Vert f\Vert_{s,\alpha,\lambda}< \infty$, we have
    \begin{equation}\label{Psi_m}
       \Lambda_1:=\sum_{\bm{k}\in\mathbb{N}_0^s\setminus K_{m,\alpha,\lambda,\beta,\bm{\gamma},\delta}}|\widehat{f}(\bm{k})|\le \Theta_{\alpha,\lambda,\beta,\bm{\gamma},\delta,\eta}\frac{\Gamma_{\alpha,\lambda,\beta,\bm{\gamma}}(m)^{\eta(\alpha+\lambda)}}{2^{\eta(\alpha+\lambda)m}}\Vert f\Vert_{s,\alpha,\lambda},
    \end{equation}
    where $\Gamma_{\alpha,\lambda,\beta,\bm{\gamma}}(m)$ is defined in (\ref{Gamma(m)}), and
    \begin{equation}\label{constant_Theta}
        \Theta_{\alpha,\lambda,\beta,\bm{\gamma},\delta,\eta}=\delta^{-\eta(\alpha+\lambda)}\sum_{u\subseteq 1:s}\gamma_{u}^{1-\frac{\eta(\alpha+\lambda)}{\alpha+\lambda+1/2}}\beta^{\eta(\alpha+\lambda)|u|} D_{\alpha,\lambda,\eta}^{|u|},
    \end{equation}
    with 
    \begin{equation}\label{constant_D}
        D_{\alpha,\lambda,\eta}=\frac{2^{-\alpha}}{\alpha!(2^{(\alpha+\lambda)(1-\eta)-1/2}-1)(2^{1-\eta}-1)^{\alpha}}+\sum_{\alpha'=0}^{\alpha}\frac{2^{-\alpha'+1}}{(\alpha')!(2^{1-\eta}-1)^{\alpha'}}.
    \end{equation}
\end{lemma}
\begin{proof}
    See Appendix \ref{appendix_psi}.
\end{proof}
\begin{remark}
    Lemma \ref{lemma_psi} implies that $\sum_{\bm{k}\in\mathbb{N}_0^s}|\widehat{f}(\bm{k})|<\infty$ whenever $\Vert f\Vert_{s,\alpha,\lambda}<\infty$ with $\alpha+\lambda>1/2$, thereby fulfilling the assumption of Lemma \ref{lemma_error_g}.
\end{remark}

\subsection{Error analysis: the main result}
By the orthonormality of the Walsh basis, the squared $L^2$ error admits the following decomposition
\begin{align}
    \Vert A^{E,\epsilon}_m(f)-f \Vert_{L^2}^2=\sum_{\bm{k}\in H^{E,\epsilon}_N}|\widehat{f}^{E,\epsilon}_{N,\med}(\bm{k})-\widehat{f}(\bm{k})|^2 + \sum_{\bm{k}\in \mathbb{N}_0^s\setminus H^{E,\epsilon}_N}|\widehat{f}(\bm{k})|^2\label{decomp_1}.
\end{align}
The first term is controlled by establishing probabilistic bounds on the coefficient estimation errors. For the second term, a decomposition analogous to \cite[Inequality ($\clubsuit$)]{pan2025universal} ensures that dominant Walsh coefficients are included in the index set $H^{E,\epsilon}_N$ with high probability. These estimates lead to our main result, formulated in Theorem \ref{thm_main} below. The complete proof is provided in Appendix \ref{appendix_main}.

\begin{theorem}\label{thm_main}
    Let $\alpha\in\mathbb{N}_0$, $\lambda\in(0,1]$, $\beta\ge 1$, $\delta\in(0,1/10)$,  $m\in\mathbb{N}$, and $N=2^m$. Assume $\alpha+\lambda>1/2$. For $f$ with $\Vert f\Vert_{s,\alpha,\lambda}<\infty$, let $\bm{\gamma}=\{\gamma_u, u\subseteq 1{:}s\}$ denote its associated relative variations \eqref{gamma_u}. Assume that the index set $K$ contains $K_{m,\alpha,\lambda,\beta,\bm{\gamma},\delta}$ defined by \eqref{K_gamma}, and that the precision $E$ satisfies (\ref{condition_E}). Furthermore, given $\delta'\in(0,1)$, let $R$ be an odd integer chosen sufficiently large so that
    \begin{equation}\label{condition_R}
        \frac{|K|+N}{2}(8\delta(1-2\delta))^{R/2} + \frac{N}{2}(20\delta(1-5\delta))^{R/2}\le \delta'.
    \end{equation}
     Then for any $\theta\in(0,1)$, with probability at least $1-\delta'$, Algorithm~\ref{alg:ep_median} satisfies
         \begin{align*}
        \Vert A^{E,\epsilon}_m(f)-f \Vert_{L^2}^2\le 
        C_{\alpha,\lambda,\beta,\bm{\gamma},\delta,\theta}\frac{ \left(\Gamma_{\alpha,\lambda,\beta,\bm{\gamma}}(m)\right)^{2\theta(\alpha+\lambda)}  }{N^{2\theta(\alpha+\lambda)+\theta-1} }\Vert f\Vert_{s,\alpha,\lambda}^2+8N\epsilon_1(\Vert f\Vert_{\infty}+\epsilon_1).
    \end{align*}
    Here, $\Gamma_{\alpha,\lambda,\beta,\bm{\gamma}}(m)$ and $\epsilon_1$ are defined in (\ref{Gamma(m)}) and (\ref{constant_ep1}), respectively, and
    \begin{align*}
        C_{\alpha,\lambda,\beta,\bm{\gamma},\delta,\theta}= &\frac{6}{\delta(1-\delta)}
        \Upsilon_{\alpha,\lambda,\beta, \bm{\gamma},\delta,\theta}
        +\frac{1}{\delta^2}\max\left\{64\Upsilon_{\alpha,\lambda,\beta, \bm{\gamma},\delta,\theta},\frac{1}{4}\Theta_{\alpha,\lambda,\beta,\bm{\gamma},\delta,\eta}^2\right\},
    \end{align*}   
    where $\Upsilon_{\alpha,\lambda,\beta, \bm{\gamma},\delta,\theta}$ and $\Theta_{\alpha,\lambda,\beta,\bm{\gamma},\delta,\eta}$ are defined in (\ref{constant_Gamma}) and (\ref{constant_Theta}), respectively, and $\eta=(1-1/(2\alpha+2\lambda))\theta$. 
\end{theorem}
\begin{proof}
    See Appendix \ref{appendix_main}.
\end{proof}

\begin{remark}\label{remark_R}
Given a failure probability tolerance $\delta'\in(0,1)$, the condition (\ref{condition_R}) is satisfied if
\begin{equation*}
    (|K|+N)(8\delta(1-2\delta))^{R/2}\le \delta'\quad\text{and}\quad N(20\delta(1-5\delta))^{R/2}\le \delta'.
\end{equation*}
A sufficient condition for these requirements to hold is
\begin{equation}\label{condition_sim_R}
    R\ge 2\max\left\{\frac{\log(\delta')-\log(|K|+N)}{\log(8\delta(1-2\delta))},\frac{\log(\delta')-\log(N)}{\log(20\delta(1-5\delta))} \right\}.
\end{equation}
In particular, if $|K|$ satisfies the growth bound
\begin{equation}\label{K_growth}
    |K|\le N(C\log N)^{s+1}
\end{equation}
for some constant $C>0$ independent of $s$, then $\log(|K|+N) = \mathcal{O}(\log N + s\log\log N)$. Consequently, for a fixed failure probability, the required number of repetitions $R$ scales as $\mathcal{O}(\log N + s\log\log N)$ under the assumption (\ref{K_growth}).

\end{remark}

\subsection{Tractability analysis}\label{sec_trac}
In this subsection, we analyze the tractability of the idealized algorithm~\ref{alg:idealized_median}, corresponding to Theorem~\ref{thm_main} with $\epsilon_1=0$ (no finite‑precision error). By Theorem~\ref{thm_main}, the error bound in terms of $N$ is independent of $s$ provided that the quantities $\Gamma_{\alpha,\lambda,\beta,\bm{\gamma}}(m)$, $\Upsilon_{\alpha,\lambda,\beta, \bm{\gamma},\delta,\theta}$, and $\Theta_{\alpha,\lambda,\beta,\bm{\gamma},\delta,\eta}$ are bounded independently of $s$. We assume the tractability condition \cite[condition (25)]{pan2026dimension}, which requires the existence of positive weights $\{\gamma_j,j\in\mathbb{N}\}$ satisfying
\begin{equation}\label{condition_gamma}
    \gamma_u\le\prod_{j\in u}\gamma_j\quad\forall u\subseteq \mathbb{N}\quad\text{and}\quad C_{\infty}:=\sum_{j=1}^{\infty}\gamma_j^{\frac{1}{\alpha+\lambda+1/2}}<\infty.
\end{equation}
Under this condition, an argument analogous to \cite[Remark 3]{pan2026dimension} shows that
\begin{equation*}
    \Upsilon_{\alpha,\lambda,\beta, \bm{\gamma},\delta,\theta}\le \delta^{-2\theta(\alpha+\lambda)}\exp(C_{\infty}\beta^{2(\alpha+\lambda)} C_{\alpha,\lambda,\theta})
\end{equation*}
with $C_{\alpha,\lambda,\theta}$ as defined in (\ref{constant_C}), and 
\begin{equation*}
    \Gamma_{\alpha,\lambda,\beta,\bm{\gamma}}(m)\le C'_{\rho}N^{\rho},\quad\forall\rho>0
\end{equation*}
with $C'_\rho$ independent of $s$. It remains to show that $\Theta_{\alpha,\lambda,\beta,\bm{\gamma},\delta,\eta}$ can be bounded independently of $s$. Since $\gamma_u\in[0,1],\beta\ge 1$, and $\eta\in (0,1-1/(2\alpha+2\lambda))$, according to (\ref{constant_Theta}), we have
\begin{align*}
  \delta^{\eta(\alpha+\lambda)}\Theta_{\alpha,\lambda,\beta,\bm{\gamma},\delta,\eta}&=\sum_{u\subseteq 1:s}\gamma_{u}^{1-\frac{\eta(\alpha+\lambda)}{\alpha+\lambda+1/2}}\beta^{\eta(\alpha+\lambda)|u|}  D_{\alpha,\lambda,\eta}^{|u|}\\
  &\le \sum_{u\subseteq 1{:}s}\gamma_{u}^{\frac{1}{\alpha+\lambda+1/2}}\beta^{(\alpha+\lambda-1/2)|u|} D_{\alpha,\lambda,\eta}^{|u|}\\
  &\le \prod_{j=1}^{s}\left(1+\gamma_j^{\frac{1}{\alpha+\lambda+1/2}}\beta^{\alpha+\lambda-1/2}  D_{\alpha,\lambda,\eta}\right)\\
  &\le \exp\left(C_{\infty}\beta^{\alpha+\lambda-1/2} D_{\alpha,\lambda,\eta}\right),
\end{align*}
where $D_{\alpha,\lambda,\eta}$ is defined in (\ref{constant_D}). Therefore, under the condition \eqref{condition_gamma}, the  probabilistic $L^2$-convergence rate can be made arbitrarily close to $\mathcal{O}(N^{-\alpha-\lambda})$ with an implied constant independent of $s$.

Finally, Remark \ref{remark_R} shows that the total number of evaluations $M = RN$ satisfies $M = \mathcal{O}((\log N + s\log\log N)N)$ under the assumption (\ref{K_growth}). Consequently, if the tractability condition (\ref{condition_gamma}) also holds, the convergence rate in terms of $M$ is 
$$\mathcal{O}((\log M/M)^{\alpha+\lambda-\eta}+(s\log\log M/M)^{\alpha+\lambda-\eta})=O(s^{\alpha+\lambda}M^{-\alpha-\lambda+\eta'}),$$
where the implied constant is independent of $s$ and $\eta'>\eta>0$ is arbitrarily small.

\section{Implementation of median digital-net $L^2$ approximation}\label{sec_4}
In this section, we discuss the implementation details of Algorithm~\ref{alg:ep_median}. We first propose two strategies for constructing the index set $K$:  one requiring known smoothness parameters $\alpha$ and $\lambda$, and the other covering all parameters. We then reduce the computational cost of Walsh coefficient estimation via the fast Walsh--Hadamard transform (FWHT). We finally develop an acceleration strategy based on Gray code ordering, tailored to the universal index set that covers all parameters.

\subsection{Construction of the index set $K$}
We present two approaches for constructing the index set $K$ required by Algorithm~\ref{alg:ep_median}.

\subsubsection{Universal index set when smoothness parameters are known}
We first consider the case where the smoothness parameters $\alpha$ and $\lambda$ are given, aiming to construct a universal index set $K^{\operatorname{univ}}_{\alpha,\lambda}$ independent of the weights $\bm{\gamma}$. 

For $\tau\ge 0$, let 
\begin{equation}\label{K_uinv^alpha}
    K^{\operatorname{univ}}_{\alpha,\lambda}(\tau)=\bigcup_{u\subseteq 1{:}s}K_{u,\alpha,\lambda}(T^{\operatorname{univ}}_{u,m,\alpha,\lambda}(\tau)),
\end{equation}
where
\begin{equation}\label{T_tau}
    T^{\operatorname{univ}}_{u,m,\alpha,\lambda}(\tau)=(\alpha+\lambda)\left[m - \log_2\left(1 + B_{\alpha,\lambda}^{|u|} \right)-\tau\sum_{j\in u}\log_2(j) \right].
\end{equation}
The following theorem shows that $K^{\operatorname{univ}}_{\alpha,\lambda}(\tau)$ satisfies the growth assumption (\ref{K_growth}), justifying the discussion in Section~\ref{sec_trac}. Furthermore, Theorem~\ref{thm_main} is applicable with $K=K^{\operatorname{univ}}_{\alpha,\lambda}(\tau)$ under suitable assumptions on the weights $\gamma_u$.

\begin{theorem}\label{thm_card_alpha}
    For $\tau\geq 0$ and $ K^{\operatorname{univ}}_{\alpha,\lambda}(\tau)$ defined in \eqref{K_uinv^alpha},
        \begin{align}\label{eqn:Kunivcard1}
        | K^{\operatorname{univ}}_{\alpha,\lambda}(\tau)|\le
            2^m\prod_{j=1}^{s}\left(1+j^{-\tau}\left(1+m\right)\right).
    \end{align}
    Furthermore, if for some $\Gamma\geq 1$, the weights $\gamma_u$ satisfy 
    \begin{equation}\label{restrict_gamma}
    \gamma_{u}\le \Gamma^{|u|} \prod_{j\in u}j^{-\tau(\alpha+\lambda+1/2)} \ \forall u\subseteq 1{:}s,
\end{equation}
     then
    $ K_{m,\alpha,\lambda,\beta,\bm{\gamma},\delta}\subseteq  K^{\operatorname{univ}}_{\alpha,\lambda}(\tau)$ with $\delta\in (0,1/2)$ and $\beta=\Gamma^{\frac{1}{\alpha+\lambda+1/2}}\max(1,B_{\alpha,\lambda})$.
\end{theorem}
\begin{proof}
    According to Lemma~\ref{lemma_K_u(T)},
    \begin{align*}
    | K^{\operatorname{univ}}_{\alpha,\lambda}(\tau)|
    &=\sum_{u\subseteq 1:s}|K_{u,\alpha,\lambda}(T^{\operatorname{univ}}_{u,m,\alpha,\lambda}(\tau))|\\
    &\le \sum_{u\subseteq 1:s}2^{T^{\operatorname{univ}}_{u,m,\alpha,\lambda}(\tau)/(\alpha+\lambda)}(A_{\alpha,\lambda}m + B_{\alpha,\lambda})^{|u|}\\
    &=2^m\sum_{u\subseteq 1:s}\frac{(A_{\alpha,\lambda}m + B_{\alpha,\lambda})^{|u|}}{1+B_{\alpha,\lambda}^{|u|}}\prod_{j\in u}j^{-\tau}.
    \end{align*}
    If $\alpha = 0$, then $A_{0,\lambda}=1,B_{0,\lambda}=0$, and
    \begin{equation*}
        \sum_{u\subseteq 1:s}\frac{(A_{\alpha,\lambda}m + B_{\alpha,\lambda})^{|u|}}{1+B_{\alpha,\lambda}^{|u|}}\prod_{j\in u}j^{-\tau}=\sum_{u\subseteq 1:s}m^{|u|}\prod_{j\in u}j^{-\tau}=\prod_{j=1}^{s}\left(1+j^{-\tau}m\right).
    \end{equation*}
    If $\alpha>0$, then $B_{\alpha,\lambda}\ge A_{\alpha,\lambda}>0$, and
    \begin{align*}
        \sum_{u\subseteq 1:s}\frac{(A_{\alpha,\lambda}m + B_{\alpha,\lambda})^{|u|}}{1+B_{\alpha,\lambda}^{|u|}}\prod_{j\in u}j^{-\tau}&< \sum_{u\subseteq 1:s}\frac{(B_{\alpha,\lambda}m + B_{\alpha,\lambda})^{|u|}}{B_{\alpha,\lambda}^{|u|}}\prod_{j\in u}j^{-\tau}\\
        &= \prod_{j=1}^{s}\left(1+j^{-\tau}\left(1+m\right)\right).
    \end{align*}
     The first assertion holds in both cases.

    Next, using (\ref{T_u_m_gamma}) and $\Gamma_{\alpha,\lambda,\beta,\bm{\gamma}}(m)\geq 1$, 
    \begin{align*}
  \frac{T_{u,m,\alpha,\lambda,\beta,\bm{\gamma},\delta}}{\alpha+\lambda} &\leq m+\log_2\delta +\frac{\log_2(\gamma_u)}{\alpha+\lambda+1/2}-|u|\log_2 \beta\\
    &\leq m+\log_2\delta-\tau\sum_{j\in u}\log_2(j)+|u|\left(\frac{\log_2\Gamma}{\alpha+\lambda+1/2}-\log_2 \beta\right)\\
    &\le \frac{T^{\operatorname{univ}}_{u,m,\alpha,\lambda}(\tau)}{\alpha+\lambda}+\log_2\left(\delta(1+B^{|u|}_{\alpha,\lambda})\right)+|u|\left(\frac{\log_2\Gamma}{\alpha+\lambda+1/2}-\log_2 \beta\right).
    \end{align*}
Hence, when $\delta\in (0,1/2)$ and $\beta=\Gamma^{\frac{1}{\alpha+\lambda+1/2}}\max(1,B_{\alpha,\lambda})$, we have 
\begin{equation}\label{T_u_m}
    T_{u,m,\alpha,\lambda,\beta,\bm{\gamma},\delta}\le T^{\operatorname{univ}}_{u,m,\alpha,\lambda}(\tau).
\end{equation}
The relation $ K_{m,\alpha,\lambda,\beta,\bm{\gamma},\delta}\subseteq  K^{\operatorname{univ}}_{\alpha,\lambda}(\tau)$ follows from \eqref{K_u}.
\end{proof}
\begin{remark}
    The parameter $\tau$ controls the trade-off between weight restrictions and index set size. When $\tau=0$,  the condition \eqref{restrict_gamma} is automatically satisfied since $\gamma_u \le 1$ by definition, representing the unrestricted case. When $\tau > 0$, it imposes decay on $\gamma_u$ and shrinks the index set. In particular, choosing $\tau > 1$ ensures that $|K^{\operatorname{univ}}_{\alpha,\lambda}(\tau)|$ is bounded independently of the dimension $s$.
\end{remark}

\subsubsection{Universal index set covering all parameters}\label{subsec:univ}
We next construct a universal index set $K^{\operatorname{univ}}$ such that $K_{m,\alpha,\lambda,\beta,\bm{\gamma},\delta}\subseteq K^{\operatorname{univ}}$ for all admissible smoothness parameters $\alpha,\lambda$ and weights $\bsgamma$, thereby removing the need for parameter specification.

For $\tau\ge 0$, let $K^{\operatorname{univ}}_{0}(\tau)=\emptyset$ and, when $\alpha\in \natu$,
\begin{equation}\label{K_univ^al}
     K^{\operatorname{univ}}_{\alpha}(\tau)=\bigcup_{u\subseteq 1{:}s}K_{u,\alpha-1,1}(T^{\operatorname{univ}}_{u,m,\alpha}(\tau)),
\end{equation}
where
\begin{equation*}
    T^{\operatorname{univ}}_{u,m,\alpha}(\tau)=\alpha\left[m -\log_2\left(1+B_{\alpha-1,0}^{|u|} \right)-\tau\sum_{j\in u}\log_2(j)\right].
\end{equation*}
Here, we use the convention $B_{0,0}=0$. The universal index set is defined by 
\begin{equation}\label{K_univ}
    K^{\operatorname{univ}}(\tau)=\bigcup_{\alpha\in\mathbb{N}}K^{\operatorname{univ}}_{\alpha}(\tau).
\end{equation}

The next theorem is the counterpart of Theorem~\ref{thm_card_alpha} for $K^{\operatorname{univ}}(\tau)$, establishing that the choice $K=K^{\operatorname{univ}}(\tau)$ satisfies the growth assumption (\ref{K_growth}) and, subject to conditions on $\gamma_u$, the assumption $K_{m,\alpha,\lambda,\beta,\bm{\gamma},\delta}\subseteq K$ used in Theorem~\ref{thm_main}.

\begin{theorem}\label{thm_card_univ}
    For $\tau\geq 0$ and $K^{\operatorname{univ}}(\tau)$ defined in \eqref{K_univ}, there exists an absolute constant $C^{\operatorname{univ}}>0$ such that for $m\geq 1$,
    \begin{equation}\label{eqn:Kunivcard2}
    |K^{\operatorname{univ}}(\tau)|\le C^{\operatorname{univ}}m2^m\prod_{j=1}^s\left(1+C^{\operatorname{univ}}j^{-\tau}(1+m)\right).    
    \end{equation}
     Furthermore, if the weights $\gamma_u$ satisfy (\ref{restrict_gamma}) for some $\alpha\in \natu_0$ and $\lambda\in (0,1]$, then  $K_{m,\alpha,\lambda,\beta,\bm{\gamma},\delta}\subseteq K^{\operatorname{univ}}(\tau)$ with $\delta\in (0,1/2)$ and $\beta=\Gamma^{\frac{1}{\alpha+\lambda+1/2}}\max(1,B_{\alpha,\lambda})$.
\end{theorem}
\begin{proof}
See Appendix~\ref{appendix_thm_card_univ}.  
\end{proof}

\subsection{Fast Estimation of Walsh Coefficients}
In this subsection, we accelerate the Walsh coefficient computation step \eqref{coe_in_alg2} in Algorithm~\ref{alg:ep_median} via FWHT. We first ignore the rounding error term $\epsilon_{\bsk,r}$ and consider an error-free counterpart
\begin{equation*}
    \widehat{f}^{E}_{N,r}(\bm{k})=\frac{1}{N}\sum_{i=0}^{N-1}f(\bm{x}_{i,r}^E)\wal{\bm{k}}(\bm{x}_{i,r}^E),
\end{equation*}
where $\bm{x}_{i,r}^E$ are the digital net points constructed in (\ref{E_digit}).

For each fixed $r$, a naive computation of all coefficients $\{\widehat{f}^{E}_{N,r}(\bm{k})\}_{\bm{k} \in K}$ would require $\mathcal{O}(|K|N)$ operations. However, these estimators actually take at most $2N$ distinct values. To illustrate this, we define the digit truncation operator
\begin{equation*}
    \trE (k)=(k_1,k_2,\ldots,k_E)^\top \quad\text{for}\quad k=\sum_{\ell=1}^\infty k_\ell 2^{\ell-1}.
\end{equation*}
From the construction \eqref{E_digit}, we have 
\begin{equation*}
    \wal{\bm{k}}(\bm{x}_{i,r}^E)=(-1)^{\sum_{j=1}^{s}\Vec{k}_j^\top \Vec{x}^E_{ij,r}}=(-1)^{\sum_{j=1}^{s}\trE(k_j)^\top C^E_{j,r}\Vec{i}}(-1)^{ \sum_{j=1}^{s}\trE(k_j)^\top\Vec{D}^E_{j,r}}.
\end{equation*}
Consequently,
\begin{equation*}
    \widehat{f}^{E}_{N,r}(\bm{k})=\frac{1}{N}(-1)^{\sum_{j=1}^{s}\trE(k_j)^\top\Vec{D}^E_{j,r}} \sum_{i=0}^{N-1}f(\bm{x}_{i,r}^E)(-1)^{(\sum_{j=1}^{s}\trE(k_j)^\top C^E_{j,r})\Vec{i}},
\end{equation*}
which means that $\widehat{f}^{E}_{N,r}(\bm{k})$ is determined by the values of $\sum_{j=1}^{s}\trE(k_j)^\top C^E_{j,r}$ and $\sum_{j=1}^{s}\trE(k_j)^\top\Vec{D}^E_{j,r}$. Since these terms belong to $\{0,1\}^{1\times m}$ and $\{0,1\}$, respectively, under modulo 2 arithmetic, we have
\begin{equation*}
    \{\widehat{f}^{E}_{N,r}(\bm{k}):\bm{k}\in K\}\subseteq B:=\left\{ \pm\frac{1}{N}\sum_{i=0}^{N-1}f(\bm{x}_{i,r}^E)(-1)^{\Vec{\ell}^\top \Vec{i}}:\ell=0,1,\ldots,N-1 \right\}.
\end{equation*}
Notably, the set $B$ has cardinality at most $2N$, and its elements can be computed simultaneously using FWHT, which leads to leads to Algorithm \ref{alg:accelerate}:

\begin{Algorithm}\label{alg:accelerate}
For given $m\in\mathbb{N}$, finite index set $K\subseteq \mathbb{N}_0^s$, precision $E\in\mathbb{N}$, generating matrices $C^E_{1,r},\ldots,C^E_{s,r}\in\{0,1\}^{E\times m}$, generating vectors $\vec{D}^E_{1,r},\ldots,\vec{D}^E_{s,r}\in \{0,1\}^{E}$, and the stored function evaluations $\{f(\bm{x}_{i,r}^E)\}_{i=0}^{N-1}$, perform the following steps:
\begin{enumerate}
    \item Apply the fast Walsh--Hadamard transform to compute
    \begin{equation}\label{eqn:fdigit}
    \widehat{f}_{\text{digit}}(\ell)= \frac{1}{N}\sum_{i=0}^{N-1}f(\bm{x}_{i,r}^E)(-1)^{\Vec{\ell}^\top \Vec{i}},\qquad \ell=0,1,\ldots,N-1.        
    \end{equation}
    
    \item For each $\bm{k}\in K$, compute
\begin{align}
    \sum_{j=1}^{s}\trE(k_j)^\top C^E_{j,r}\mod{2}&=\vec{\ell}^\top \quad\text{for}\quad \ell\in\{0,1,\ldots,N-1\},\label{comp_kC}\\
    \sum_{j=1}^{s}\trE(k_j)^\top\Vec{D}^E_{j,r}\mod{2}&= a\quad\text{for}\quad a\in\{0,1\}\label{comp_kD}
\end{align}
and set
\begin{equation}\label{eqn:copy}
 \widehat{f}^{E}_{N,r}(\bm{k})=(-1)^a \widehat{f}_{\text{digit}}(\ell).   
\end{equation}

    \item Return the values $\{\widehat{f}^{E}_{N,r}(\bm{k}) : \bm{k}\in K\}$.
\end{enumerate}
\end{Algorithm}

The total computational cost of Algorithm~\ref{alg:accelerate} comprises two components. First, applying FWHT requires $\mathcal{O}(N\log N)$ operations. Second, for each $\bm{k}\in K$, determining the index $\ell$ and the sign $(-1)^a$ using modulo 2 arithmetic requires $\mathcal{O}(sEm)$ operations. Consequently, the cost of processing the entire set $K$ is $\mathcal{O}(|K|sEm)$. Combining these two parts and substituting $m = \log_2 N$, the overall computational complexity of Algorithm~\ref{alg:accelerate} is $\mathcal{O}(N\log N + |K|sE\log N)$, significantly reducing the $\mathcal{O}(|K|N)$ operations required by the naive approach.

When rounding errors are present, our analysis remains the same except that \eqref{eqn:fdigit} yields $\widehat{f}_{\text{digit}}(\ell)$ plus a rounding error $\epsilon_\ell$. This error (or its negative) constitutes the error term $\epsilon_{\bsk,r}$ as the contaminated $\widehat{f}_{\text{digit}}(\ell)$ is copied into $\widehat{f}^{E}_{N,r}(\bm{k})$ in \eqref{eqn:copy}.

\subsection{Gray code ordering for $K^{\operatorname{univ}}(\tau)$}\label{subsubsec_ac}

The universal set $K^{\operatorname{univ}}(\tau)$ introduced in Section~\ref{subsec:univ}  has a structure that allows further acceleration of Walsh coefficient estimation. By partitioning the set into subsets and constructing Gray codes \cite{walsh2003generating} for each, with elements ordered so that consecutive $\bsk$ differ in only a few components,  the computations in \eqref{comp_kC} and \eqref{comp_kD} can be carried out via incremental updates.
 
We first consider index sets of the form
\begin{equation*}
    K_{u,\alpha-1,1}(T) = \left\{\bm{k}\in \mathbb{N}_0^s:\bm{s}(\bm{k})=u,\sum_{j\in u}\mu_{\alpha-1,1}(k_j)\le T \right\}.
\end{equation*}
For non-empty $u\subseteq 1{:}s$, we can view the elements of $K_{u,\alpha-1,1}(T)$ as elements of $\mathbb{N}^{|u|}$. Specifically, for $\bm{k}=(k_1,\ldots,k_{|u|})\in K_{u,\alpha-1,1}(T)$, we have the bijection
\[
\bm{k}\longleftrightarrow(\bm{w}_1,\bm{v}_1;\ldots;\bm{w}_{|u|},\bm{v}_{|u|}),
\]
where, for each $j=1,\ldots,|u|$, the pair $(\bm w_j,\bm v_j)$ is uniquely determined by the binary expansion
\[
k_j=\sum_{\ell\in\mathbb N}\kappa_{j,\ell}2^{\ell-1},\qquad \kappa_{j,\ell}\in\mathbb \{0,1\},
\]
as follows. Let
\[
z_j:=\min\!\Bigl(\alpha,\sum_{\ell\in\mathbb N}\kappa_{j,\ell}\Bigr),
\]
and let
\[
w_{j,1}>\cdots>w_{j,z_j}\ge1
\]
be the positions of the first $z_j$ nonzero binary digits of $k_j$. Setting
\[
w_{j,z_j+1}=\cdots=w_{j,\alpha}=0,
\]
we define
\[
\bm w_j=(w_{j,1},\ldots,w_{j,\alpha})\in\mathbb N_0^\alpha.
\]
Further, when $w_{j,z_j}>1$, let
\[
\bm v_j=(\kappa_{j,1},\ldots,\kappa_{j,w_{j,z_j}-1})\in\{0,1\}^{\,w_{j,z_j}-1}.
\]
When $w_{j,z_j}=1$, we set $\bm v_j$ to be the null vector $\{0,1\}^0$.

Further denoting
\begin{equation}\label{R_w_eq}
R_j = \mu_{\alpha-1,1}(k_j)=\sum_{\ell = 1}^{z_j}w_{j,\ell}\in \mathbb{N}, 
\end{equation}
we then have $\bm{k}\in K_{u,\alpha-1,1}(T)$ if and only if 
\begin{equation}\label{restrction_w}
    0\le \sum_{j=1}^{|u|}(R_j-1)\le T-|u|.
\end{equation}

We now order all elements in $K_{u,\alpha-1,1}(T)$ as follows.
\begin{enumerate}

\item \textbf{First level:} The base tuples $(R_1, \dots, R_{|u|})$ satisfying \eqref{restrction_w} are arranged in an arbitrary order, for example, the lexicographical order.

\item \textbf{Second level:} For each fixed tuple $(R_1,\ldots,R_{|u|})$, consider all extended solutions $\bm w_j=(w_{j,1},\ldots,w_{j,\alpha})\in\mathbb N_0^\alpha$ of \eqref{R_w_eq}, ordered so that $(R_j,0,\ldots,0)$ appears first, while $(R_j-1,1,0,\ldots,0)$ appears last whenever \eqref{R_w_eq} admits at least two solutions; all remaining admissible solutions may be arranged arbitrarily, for instance in lexicographical order. The tuples $(\bm w_1,\ldots,\bm w_{|u|})$ are then ordered lexicographically with respect to this prescribed ordering of each component $\bm w_j$.

\item \textbf{Third level:} For each fixed tuple $(\bm w_1,\ldots,\bm w_{|u|})$, each block $\bm v_j$ is ordered as follows: if $z_j=\alpha$, then $\bm v_j\in \{0,1\}^{\,w_{j,\alpha}-1}$ is ordered according to the standard binary Gray code $i\mapsto i\oplus\lfloor i/2\rfloor$, $i=0,1,\ldots,2^{\,w_{j,\alpha}-1}-1$, where each codeword is read with the least significant bit as the first coordinate; if $z_j<\alpha$, then $\bm v_j$ is the unique zero vector $\{0,1\}^{\,w_{j,z_j}-1}$. The tuples $(\bm v_1,\ldots,\bm v_{|u|})$ are then ordered lexicographically with respect to the prescribed ordering of each component $\bm v_j$.

\end{enumerate}

\begin{proposition}\label{prop_gray_code}
    Let $(\bm{w}_j^i; \bm{v}_j^i)_{j=1}^{|u|}$ for $i \in \{1,2\}$ be two consecutive elements in $K_{u,\alpha-1,1}(T)$ under the proposed ordering. Then
    \begin{enumerate}
        \item $(\bm{w}_1^1, \dots, \bm{w}_{|u|}^1)$ and $(\bm{w}_1^2, \dots, \bm{w}_{|u|}^2)$ differ in at most $2|u|+\alpha-1$ positions.
    
        \item For each $1\le j \le |u|$, $\bm{v}^1_j$ and $\bm{v}^2_j$ differ in at most one bit (under zero-padding to equal length).

    \end{enumerate}
\end{proposition}

\begin{proof}
See Appendix \ref{appendix_gray_code}.
\end{proof}

\begin{remark}
In the special case $\alpha=1$, each $\bm{w}_j = R_j$ reduces to a scalar, and a combinatorial Gray code can be applied to further minimize the differences between consecutive sequences $(\bm{w}_1, \dots, \bm{w}_{|u|})$. Specifically, \eqref{restrction_w} defines an $n$-composition of the interval $[0, T-|u|]$ studied in \cite{vajnovszki2011restricted}. By \cite[Corollary 2]{vajnovszki2011restricted}, these tuples can be ordered so that any two consecutive tuples differ in at most two positions, and by $+1$ or $-1$ in those positions.
\end{remark}

Proposition \ref{prop_gray_code} ensures that consecutive tuples in $K_{u,\alpha-1,1}(T)$ differ in at most $3|u|+\alpha-1$ positions. Consequently, evaluating \eqref{comp_kC} and \eqref{comp_kD} for adjacent elements requires only $\mathcal{O}((\alpha+|u|) m)$ incremental operations. For a fixed $u$, the specific ordering implies that the initial element satisfies $\bm{w}_j = (R_j, 0, \dots, 0)$ and $\bm{v}_j = \bm{0}$, so the initial evaluation cost is $\mathcal{O}(|u|m)$.  The overall complexity for evaluating \eqref{comp_kC} and \eqref{comp_kD} for all $\bsk \in K_{u,\alpha-1,1}(T)$ is thus
$\mathcal{O}((\alpha+|u|)m |K_{u,\alpha-1,1}(T)|)$. 

Next, we connect the above analysis to the universal set $K^{\operatorname{univ}}(\tau)$. Recall that $K^{\operatorname{univ}}(\tau)=\bigcup_{\alpha\in\mathbb{N}}K^{\operatorname{univ}}_{\alpha}(\tau)$, where $K^{\operatorname{univ}}_\alpha(\tau)$ from \eqref{K_univ^al} is a union over $u\subseteq 1{:}s$ of sets of the form $K_{u,\alpha-1,1}(T)$. Summing over all non-empty subsets $u \subseteq \{1:s\}$, the total complexity for $K^{\operatorname{univ}}_\alpha(\tau)$ is therefore bounded by $\mathcal{O}((\alpha+s)m|K^{\operatorname{univ}}_\alpha(\tau)|)$.

Finally, we choose a truncation parameter $\alpha_0 \in \mathbb{N}$ and split $K^{\operatorname{univ}}(\tau)$ into two parts. For $\bm{k} \in K^{\operatorname{univ}}(\tau) \setminus \bigcup_{\alpha \le \alpha_0} K^{\operatorname{univ}}_\alpha(\tau)$, we evaluate \eqref{comp_kC} and \eqref{comp_kD} directly. For those in $K^{\operatorname{univ}}_{\leq\alpha_0}(\tau)=\bigcup_{\alpha \le \alpha_0} K^{\operatorname{univ}}_\alpha(\tau)$, we apply the preceding ordering strategy, reducing the cost from $\mathcal{O}\left(sEm \left| K^{\operatorname{univ}}_{\leq\alpha_0}(\tau) \right|\right)$ to $\mathcal{O}\left((\alpha_0+s)m \left| K^{\operatorname{univ}}_{\leq\alpha_0}(\tau) \right|\right)$.   This reduction is often significant, since $E$ is typically set large to satisfy condition (\ref{condition_E}), while setting $\alpha_0=2$ already captures most indices in $K^{\operatorname{univ}}(\tau)$ (see Section~\ref{subsec:Ksize}).

\section{Numerical experiments}\label{sec_5}
In this section, we validate the proposed algorithm through numerical experiments, where $N=2^m$ denotes the number of digital net points. We first analyze the cardinality of the pre-selected index set $K$ under two different selection strategies. Afterwards, we assess the approximation performance for non-periodic functions to confirm the theoretical error bounds.

Throughout this section, we consider the universal index sets $K^{\operatorname{univ}}_{\alpha,\lambda}(\tau)$, $K^{\operatorname{univ}}_{\alpha}(\tau)$,
and
$K^{\operatorname{univ}}(\tau)$,
introduced in
\eqref{K_uinv^alpha},
\eqref{K_univ^al},
and
\eqref{K_univ},
respectively.

\subsection{Comparison of the index sets}\label{subsec:Ksize}

\begin{figure}[htbp]
  \centering
  \includegraphics[width=\linewidth]{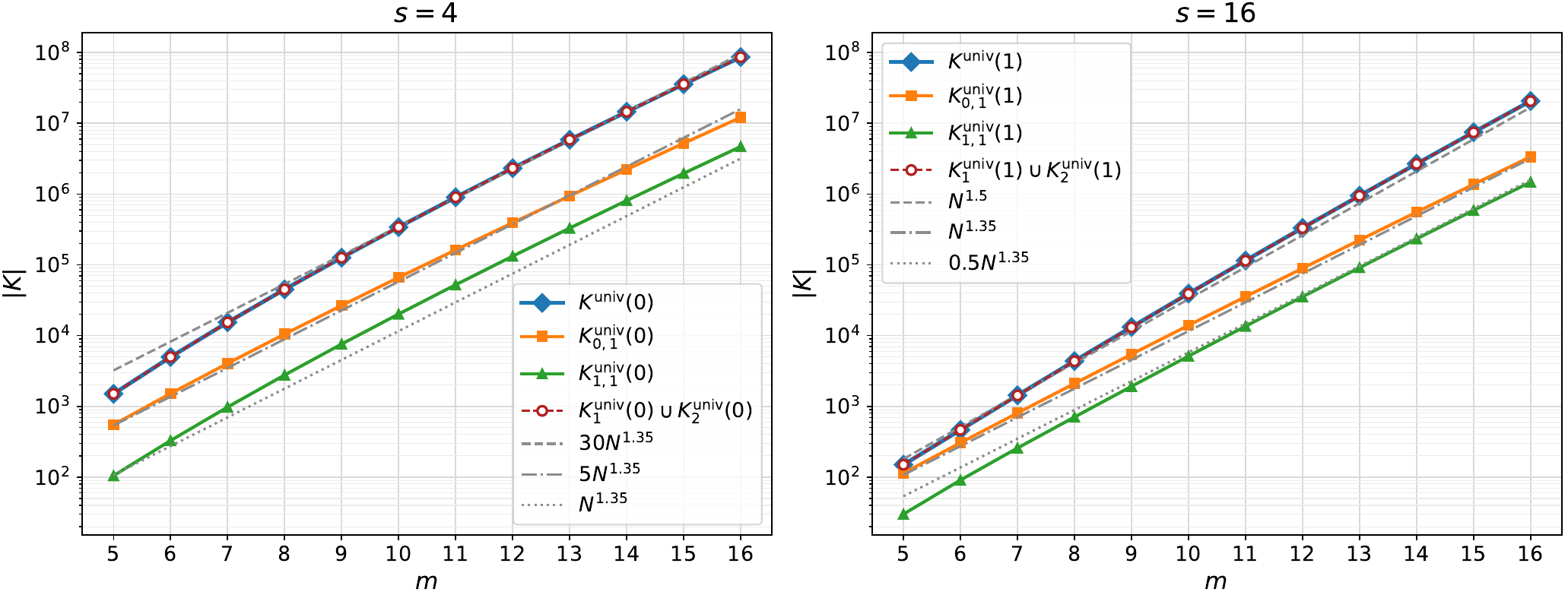}
  \caption{Cardinalities of the pre-selected index sets $K$ in dimensions $s=4$ (left) and $s=16$ (right).}
  \label{fig0:sizeK} 
\end{figure}

We investigate the cardinality of index sets $K$ by comparing $K^{\operatorname{univ}}(\tau)$ with the associated sets $ K^{\operatorname{univ}}_{\alpha,\lambda}(\tau)$ for various parameters. Figure \ref{fig0:sizeK} reports these cardinalities for two representative dimensional settings. For $s=4$, we consider the unweighted case ($\tau=0$) and compare the unrestricted sets $K^{\operatorname{univ}}(0)$, $K^{\operatorname{univ}}_{0,1}(0)$, $K^{\operatorname{univ}}_{1,1}(0)$, and the union $K^{\operatorname{univ}}_{1}(0)\cup K^{\operatorname{univ}}_{2}(0)$. In contrast, for $s=16$, the unrestricted set $K^{\operatorname{univ}}(0)$ becomes excessively large; we therefore impose the weight constraint \eqref{restrict_gamma} with decay parameter $\tau=1$ and compare the corresponding modified sets $K^{\operatorname{univ}}(1)$, $K^{\operatorname{univ}}_{0,1}(1)$, $K^{\operatorname{univ}}_{1,1}(1)$, and $K^{\operatorname{univ}}_{1}(1)\cup K^{\operatorname{univ}}_{2}(1)$.

The empirical results show that for $s=4$, $K^{\operatorname{univ}}(0)$, $K^{\operatorname{univ}}_{0,1}(0)$, and $K^{\operatorname{univ}}_{1,1}(0)$ all exhibit $\mathcal{O}(N^{1.35})$ growth rates, while for $s=16$, $K^{\operatorname{univ}}(1)$ shows $\mathcal{O}(N^{1.5})$  growth, and $K^{\operatorname{univ}}_{0,1}(1)$, $K^{\operatorname{univ}}_{1,1}(1)$ maintain $\mathcal{O}(N^{1.35})$ growth. The superlinear growth is likely due to the polynomial terms in $m=\log_2 N$ in the upper bounds \eqref{eqn:Kunivcard1} and \eqref{eqn:Kunivcard2}. Notably, the union 
$K^{\operatorname{univ}}_{1}(\tau) \cup K^{\operatorname{univ}}_{2}(\tau)$ already includes nearly all of $K^{\operatorname{univ}}(\tau)$, numerically justifying the choice $\alpha_0=2$ for the implementation in Section~\ref{subsubsec_ac}.

\subsection{Approximation of non-periodic functions}\label{subsec_num_app}
To validate our theoretical findings, we consider the following non-periodic test functions for $s=4,16$:
\begin{align*}
    f_{1}(\bm{x})&=\prod_{j=1}^{s}\left(1+\frac{1}{j^3}\left(\sqrt{x_j}-\frac{2}{3}\right)\right),\\
    f_{2}(\bm{x})&=\prod_{j=1}^{s}\left(1+\frac{1}{j^3}(x_j\exp(x_j)-1) \right).
\end{align*}
Because $\sqrt{x}$ is $1/2$-H\"older continuous and $x\exp(x)-1$ is analytic on $[0,1]$, we have 
$\Vert f_{1}\Vert_{s,0,1-\varepsilon}<\infty$ for any $\varepsilon\in (0,1)$ and $\Vert f_{2}\Vert_{s,\alpha,\lambda}<\infty$ for any $\alpha\in\mathbb{N}_0,\lambda\in(0,1]$. 
Furthermore, as the products $\prod_{j\in u}(\sqrt{x_j}-2/3)$ and $\prod_{j\in u}(x_j\exp(x_j)-1)$ for subsets $u\subseteq 1{:}s$ are mutually $L^2$-orthogonal, we have $\gamma_u=\prod_{j\in u}j^{-3}$ for both $f_1$ and $f_2$.

We set $E=60$ in Algorithm \ref{alg:ep_median}. As in Section~\ref{subsec:Ksize}, we use $\tau=0$ for $s=4$ and $\tau=1$ for $s=16$. For $f_1$, we test $K^{\operatorname{univ}}(\tau)$ and $K^{\operatorname{univ}}_{0,1}(\tau)$, and for $f_2$, we additionally test $K^{\operatorname{univ}}_{1,1}(\tau)$. The number of repetitions $R$ is chosen to be the smallest odd integer satisfying condition (\ref{condition_sim_R}) with $\delta=0.02$ and failure probability $\delta'=0.001$. The total number of function evaluations is $M=RN$. We assess the approximation accuracy of the algorithm by the \textbf{normalized $L^2$-error}:
\begin{equation*}
\frac{\Vert A_m^{E,\epsilon}(f)-f\Vert_{L^2}}{\Vert f\Vert_{L^2}}.
\end{equation*}

As a benchmark, we compare our method with the deterministic multiple-shift lattice algorithm based on the tent transformation proposed in \cite[Section~7]{du2026deterministic}. Following the parameter setting therein, we choose the smoothness parameter $\alpha=1$ for $f_1$ and $\alpha=3/2$ for $f_2$, consistent with the regularity
\[
f_1\in \mathcal{C}_{s,1-\varepsilon,\gamma},\qquad
f_2\in \mathcal{C}_{s,3/2-\varepsilon,\gamma},
\]
for arbitrarily small $\varepsilon>0$, where $\mathcal{C}_{s,\alpha,\gamma}$ denotes the weighted half-period cosine space defined in \cite[Definition~2]{du2026deterministic}. In both cases, the product weights are chosen as $ \gamma_u=\prod_{j\in u}j^{-3}.$
We employ the adaptive deterministic shift construction described in \cite[Section~5.3]{du2026deterministic}. The implementation is based on the publicly available source code at
\url{https://github.com/Jiarui-Du/multiple-shift},
and the corresponding normalized errors are labeled as \emph{Tent} in the figures.

To further evaluate the quality of the constructed index set $H_N^{E,\epsilon}$ for different choices of $K$, we also compute the associated \textbf{normalized truncation error}:
\begin{equation*}
\left(\frac{\sum_{\bm{k}\in\mathbb{N}_0^s\setminus H_N^{E,\epsilon}}|\widehat{f}(\bm{k})|^2}{\Vert f\Vert_{L^2}^2}\right)^{1/2}
=
\left(1-\frac{\sum_{\bm{k}\in H_N^{E,\epsilon}}|\widehat{f}(\bm{k})|^2}{\Vert f\Vert_{L^2}^2}\right)^{1/2},
\end{equation*}
which forms one component of the total error according to (\ref{decomp_1}), where $\widehat{f}(\bm{k})$ denotes the $\bm{k}$-th Walsh coefficient. We compare this quantity with the theoretically optimal oracle truncation error, corresponding to the set $H_N^{\mathrm{oracle}}$ consisting of the indices of the $N$ Walsh coefficients with the largest magnitudes, i.e., $|\widehat{f}(\bm{k})|\ge |\widehat{f}(\bm{h})|$ for any $\bm{k}\in H_N^{\mathrm{oracle}}$ and $\bm{h}\notin H_N^{\mathrm{oracle}}$. The corresponding normalized error is
\begin{equation*}
\left(1-\frac{\sum_{\bm{k}\in H_N^{\mathrm{oracle}}}|\widehat{f}(\bm{k})|^2}{\Vert f\Vert_{L^2}^2}\right)^{1/2}.
\end{equation*}

\begin{figure}[htbp]
  \centering
  \includegraphics[width=\linewidth]{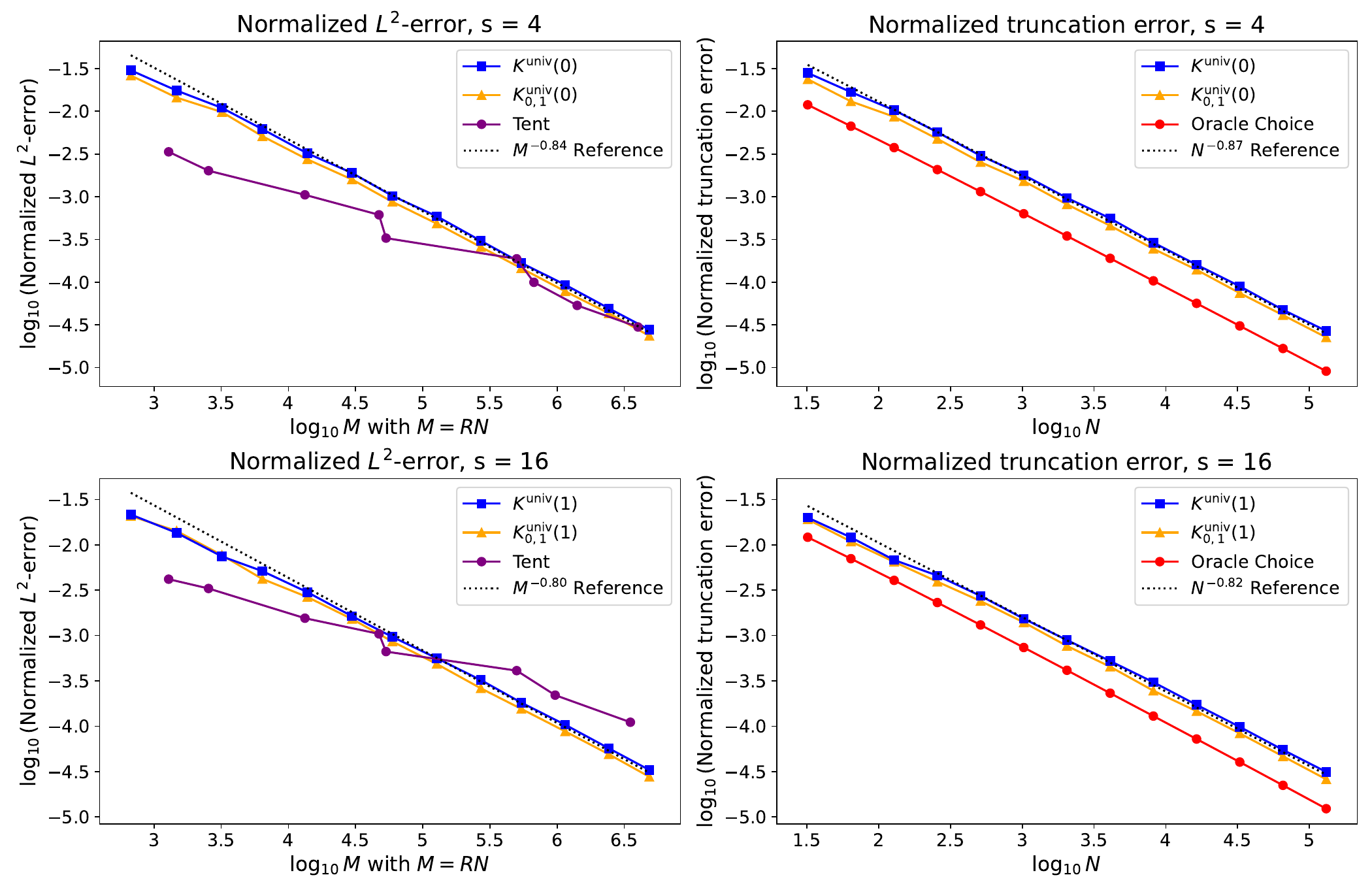}
  \caption{Normalized $L^2$-error and normalized truncation error for $f_1$ with $s = 4$ (top) and $s = 16$ (bottom).}
  \label{fig1:f_1_plt} 
\end{figure}

\begin{figure}[htbp]
  \centering
  \includegraphics[width=\linewidth]{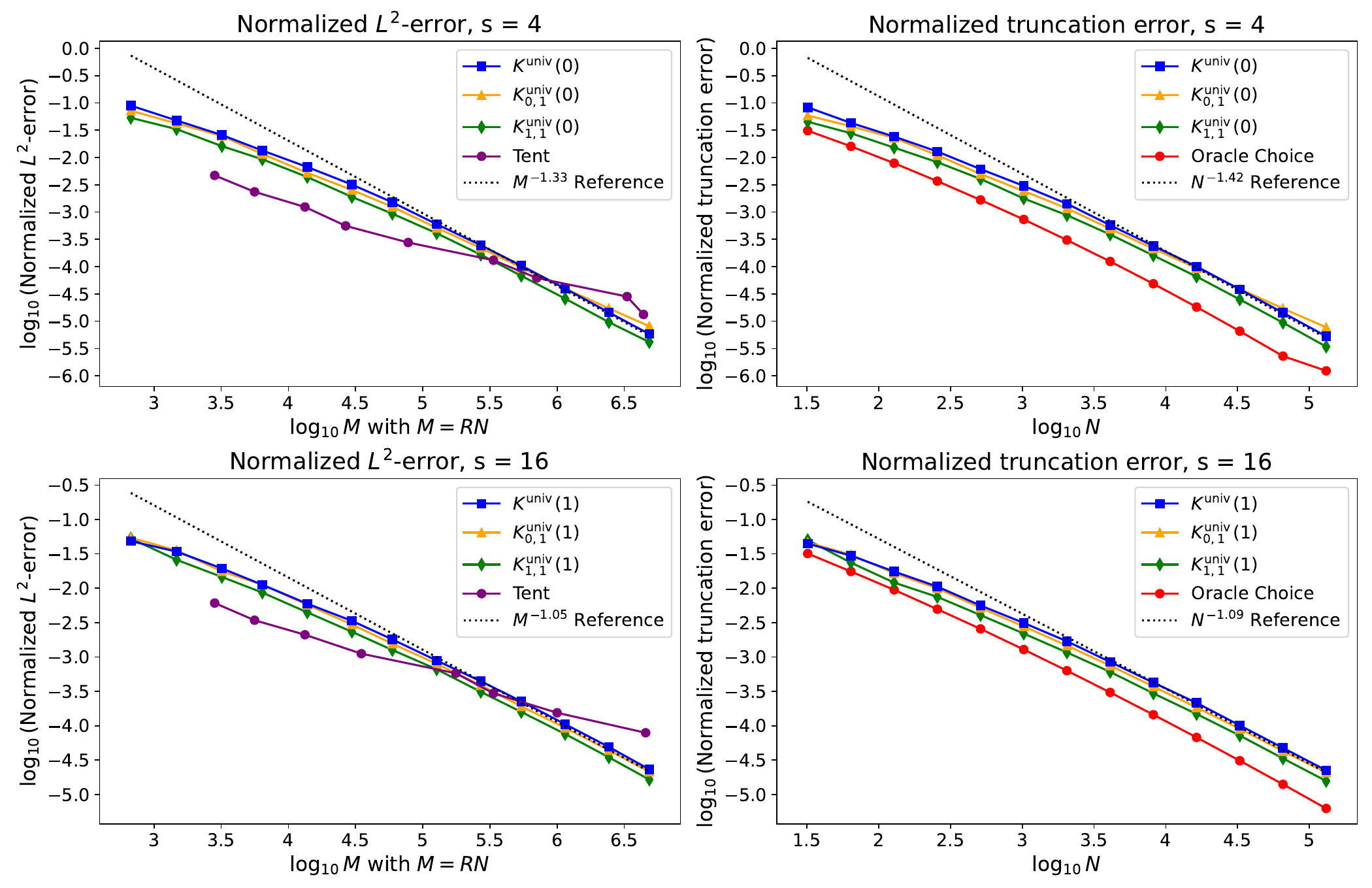}
  \caption{Normalized $L^2$-error and normalized truncation error for $f_2$ with $s = 4$ (top) and $s = 16$ (bottom).}
  \label{fig2:f_2_plt} 
\end{figure}

Figure \ref{fig1:f_1_plt} shows the normalized $L^2$-errors and normalized truncation errors for the approximation of $f_1$. In the low-dimensional case ($s=4$), the proposed method achieves an observed convergence rate of $\mathcal{O}(M^{-0.84})$, while for $s=16$ the observed rate remains $\mathcal{O}(M^{-0.80})$, indicating that the convergence behavior is robust to dimensionality. 
Compared with the tent-transformation benchmark, the proposed method exhibits larger errors for small values of $M$, likely because $f_1$ can be well approximated by the low-frequency cosine series used in the benchmark. As $M$ increases, however, the gap gradually decreases, and for $s=16$ the proposed method yields smaller errors. The proposed method also exhibits a marginally faster convergence rate over the tested range, although the theoretical asymptotic rates are both $O(M^{-1+\epsilon})$. We note that, unlike the benchmark method, the proposed algorithm requires no weight parameters and, when $K_{\operatorname{univ}}$ is used, no smoothness parameters.

A comparison of the total errors (left plots) with the truncation errors (right plots) demonstrates that the two remain close for all choices of $M$ and $K$, suggesting that the truncation errors dominate the approximation errors. Moreover, these truncation errors are within a constant factor of those under the oracle choice, indicating that $H_N^{E,\epsilon}$ effectively captures the dominant Walsh coefficients, and no significant improvement is possible when the approximant is restricted to $N$-term Walsh series.

Figure \ref{fig2:f_2_plt} reports the normalized $L^2$-errors and normalized truncation errors for the analytic function $f_2$. In the low-dimensional case ($s=4$), the observed convergence rate of the proposed method reaches $\mathcal{O}(M^{-1.33})$ at the upper end of our experimental range. As the dimensionality increases to $s=16$, the observed convergence rate decreases to $\mathcal{O}(M^{-1.05})$. As in the previous example, the proposed method achieves approximation errors comparable to the benchmark, with eventually smaller errors for sufficiently large $M$. We also observe that truncation errors under all choices of $K$ remain within a constant factor of the optimal ones, indicating that the gap between the practical convergence rate and the theoretical asymptotic rate is due to finite-sample effects.

Finally, to demonstrate how well the proposed index sets $K$ capture the dominant Walsh coefficients, Tables \ref{tab:intersection1} and \ref{tab:intersection2} report the intersection size $|H_N^{\text{oracle}} \cap K|$ for $f_1$ and $f_2$, respectively. All constructed sets $K$ capture a vast majority—and in the case of $K^{\operatorname{univ}}(\tau)$, the entirety—of the indices in $H_N^{\text{oracle}}$, explaining why the approximation errors remain close to the oracle bound. Table~\ref{tab:intersection2} also explains why $K^{\operatorname{univ}}_{1,1}(\tau)$ performs the best in Figure~\ref{fig2:f_2_plt}: $K^{\operatorname{univ}}_{1,1}(\tau)$ successfully captures all dominant Walsh coefficients for $m\geq 10$ while proposing fewer candidates than $K^{\operatorname{univ}}(\tau)$, reducing the noise in selecting $H_N^{E,\epsilon}$. These findings suggest that smoothness-specific choices of $K$ remain effective in finite-sample settings even when the target functions are much smoother.

\begin{table}[htbp]
  \centering
  \caption{Comparison of $|H_N^{\mathrm{oracle}} \cap K|$ for $f_1$ across different index sets $K$ and levels $m$. Recall that $\tau = 0$ for $s = 4$ and $\tau = 1$ for $s = 16$.}
  \label{tab:intersection1}
  \begin{tabular}{cc cccc cccc}
    \toprule
    \multirow{2}{*}{$m$} & \multirow{2}{*}{$N=2^m$} & \multicolumn{2}{c}{$K^{\operatorname{univ}}_{0,1}(\tau)$}  && \multicolumn{2}{c}{$K^{\operatorname{univ}}(\tau)$} \\
    \cmidrule{3-4}\cmidrule{5-6}
    & & $s=4$ & $s = 16$ && $s=4$ & $s = 16$ \\
    \midrule
    8 & 256 & 246   & 246 && \textbf{256}  & \textbf{256} \\
    10 & 1024 & 998  & 999 && \textbf{1024}   & \textbf{1024} \\
    12 & 4096 & 4028 & 4031 && \textbf{4096}  & \textbf{4096} \\
    14 & 16384 & 16169 & 16212 && \textbf{16384} & \textbf{16384} \\
    16 & 65536 & 64955 & 65019 && \textbf{65536} & \textbf{65536}\\
    \bottomrule
  \end{tabular}
\end{table}

\begin{table}[htbp]
  \centering
  \caption{Comparison of $|H_N^{\mathrm{oracle}} \cap K|$ for $f_2$ across different index sets $K$ and levels $m$. Recall that $\tau = 0$ for $s = 4$ and $\tau = 1$ for $s = 16$.}
  \label{tab:intersection2}
  \begin{tabular}{cc cccc cccc cccc}
    \toprule
    \multirow{2}{*}{$m$} & \multirow{2}{*}{$N=2^m$} & \multicolumn{2}{c}{$K^{\operatorname{univ}}_{0,1}(\tau)$} && \multicolumn{2}{c}{$K^{\operatorname{univ}}_{1,1}(\tau)$} && \multicolumn{2}{c}{$K^{\operatorname{univ}}(\tau)$} \\
    \cmidrule{3-4} \cmidrule{6-7} \cmidrule{9-10}
    & & $s=4$ & $s = 16$ && $s=4$ & $s = 16$ && $s=4$ & $s = 16$ \\
    \midrule
    8 & 256 & 242   & 243   && \textbf{256}   & 253   && \textbf{256}  & \textbf{256} \\
    10 & 1024 & 991  & 971  && \textbf{1024}   & \textbf{1024}  && \textbf{1024}   & \textbf{1024} \\
    12 & 4096 & 3931 & 3885  && \textbf{4096}  & \textbf{4096}  && \textbf{4096}  & \textbf{4096} \\
    14 & 16384 & 15630 & 15455 && \textbf{16384}  & \textbf{16384} && \textbf{16384} & \textbf{16384} \\
    16 & 65536 & 62340 & 62044 && \textbf{65536} & \textbf{65536} && \textbf{65536} & \textbf{65536}\\
    \bottomrule
  \end{tabular}
\end{table}

\section{Conclusion}\label{sec_6}
In this paper, we propose a universal median digital-net algorithm for the $L^2$-approximation of non-periodic functions on $[0,1]^s$, requiring no weight parameters as input. We prove that the algorithm achieves a probabilistic error bound of $\mathcal{O}(M^{-\alpha-\lambda+\eta})$ for any $\eta>0$ under the smoothness assumption $\Vert f\Vert_{s,\alpha,\lambda}<\infty$, and the implied constant grows at most polynomially in $s$ under suitable decay conditions of the relative variations $\gamma_u$ and an appropriate choice of the pre-selected index set $K$. We propose two constructions of $K$ for practical implementation, one tailored to known smoothness parameters $\alpha, \lambda$ and one covering all parameters. We also employ the FWHT algorithm and Gray code ordering to accelerate the estimation of Walsh coefficients. 

Some challenges remain to be addressed in future research. First, as discussed in Section \ref{subsec_num_app}, the choice of index set $K$ is crucial for the practical performance. Removing insignificant coefficients in $K$ not only accelerates the implementation, but also reduces the selection noise and leads to smaller approximation error. Second, both Figure~\ref{fig1:f_1_plt} and~\ref{fig2:f_2_plt} demonstrate that the truncation error is dominant when using $N$-term Walsh series as the approximant,  suggesting that increasing the number of terms may reduce the approximation error. How to optimally tune this parameter however remains an open question. Lastly, our algorithm produces discontinuous approximants for continuous target functions, making the approximation poor in derivative-based norms. Extending our framework to smoother wavelet bases is therefore another important direction for future investigation.

\appendix

\section*{Appendix Overview} 
The appendices are organized as follows. Appendix~\ref{appendix_lambda} proves Lemma~\ref{lemma_lambda}, Appendix~\ref{appendix_psi} proves Lemma~\ref{lemma_psi}, Appendix~\ref{appendix_main} contains the proof of the main result, Theorem~\ref{thm_main},
Appendix~\ref{appendix_thm_card_univ} contains the proof of Theorem~\ref{thm_card_univ},
and Appendix~\ref{appendix_gray_code} establishes Proposition~\ref{prop_gray_code}. Throughout Appendix~\ref{appendix_lambda} to \ref{appendix_main}, the parameters $\alpha, \lambda, \delta$, and $\beta$ are treated as fixed constants. To streamline the derivations, we suppress the explicit dependence on these parameters in certain intermediate quantities. Specifically, we adopt the following shorthand notation
\begin{align*}
    K_{u,\alpha,\lambda}(T) = K_u(T),\quad T_{u,m,\alpha,\lambda,\beta,\bm{\gamma},\delta} = T_{u,m,\bm{\gamma}},\quad K_{m,\alpha,\lambda,\beta,\bm{\gamma},\delta} = K_{m,\bm{\gamma}},
\end{align*}
where the full expressions are defined in (\ref{K_u}), (\ref{T_u_m_gamma}), and (\ref{K_gamma}), respectively.

\section{Proof of Lemma \ref{lemma_lambda}}\label{appendix_lambda}
\begin{proof}
    According to \cite[Lemma 5]{pan2026dimension}, for $\theta\in (0,1)$ and $\emptyset\ne u\subseteq 1{:}s$, we have
    \begin{equation}\label{ap_L2_K_u}
        \sum_{\bm{k}\in\mathbb{N}_*^s\setminus K_{u}(T)}|\widehat{f}_u(\bm{k})|^2\le \frac{\Vert f\Vert_{u,\alpha,\lambda}^2}{4^{\theta T}}C^{|u|}_{\alpha,\lambda,\theta}.
    \end{equation}
    From (\ref{wal_f_u}) we have $\widehat{f}(\bm{k})=\widehat{f}_u(\bm{k})$ if $\bm{s}(\bm{k})=u$. Therefore,
    \begin{align}
        \sum_{\bm{k}\in\mathbb{N}_0^s\setminus K_{m,\bm{\gamma}}}|\widehat{f}(\bm{k})|^2&=\sum_{u\subseteq 1{:}s}\sum_{\substack{\bm{k}\in\mathbb{N}_0^s\setminus K_{m,\bm{\gamma}}\\ \bm{s}(\bm{k})=u}}|\widehat{f}(\bm{k})|^2=\sum_{\emptyset\ne u\subseteq 1{:}s}\sum_{\bm{k}\in\mathbb{N}_*^s\setminus K_u(T_{u,m,\bm{\gamma}})}|\widehat{f}_u(\bm{k})|^2\nonumber\\
        &\le \sum_{u\subseteq 1{:}s} \frac{\Vert f\Vert_{u,\alpha,\lambda}^2}{4^{\theta T_{u,m,\bm{\gamma}}}}C^{|u|}_{\alpha,\lambda,\theta}\label{ap_L2_K_r_1}\\
        &= \frac{\Gamma_{\alpha,\lambda,\beta,\bm{\gamma}}(m)^{2\theta(\alpha+\lambda)}}{\delta^{2\theta(\alpha+\lambda)}2^{2\theta(\alpha+\lambda)m}}\sum_{u\subseteq 1{:}s}\gamma_{u}^{-\frac{2\theta(\alpha+\lambda)}{\alpha+\lambda+1/2}}\beta^{2\theta(\alpha+\lambda)|u|} C^{|u|}_{\alpha,\lambda,\theta}\gamma_u^2\Vert f\Vert^2_{s,\alpha,\lambda}\label{ap_L2_K_r_2}\\
        &= \Gamma_{\alpha,\lambda,\beta,\bm{\gamma},\delta,\theta}\frac{\Gamma_{\alpha,\lambda,\beta,\bm{\gamma}}(m)^{2\theta(\alpha+\lambda)}}{2^{2\theta(\alpha+\lambda)m}}\Vert f\Vert_{s,\alpha,\lambda}^2,\nonumber
    \end{align}
    where (\ref{ap_L2_K_r_1}) used (\ref{ap_L2_K_u}) and (\ref{ap_L2_K_r_2}) used the definition of $\gamma_u$ in (\ref{gamma_u}).
\end{proof}

\section{Proof of Lemma \ref{lemma_psi}}\label{appendix_psi}
To prove Lemma \ref{lemma_psi}, we need to establish the $L^1$-versions of \cite[Lemma 4]{pan2026dimension} and \cite[Lemma 5]{pan2026dimension}. We first introduce the following local notation. 

For $k\in\mathbb{N}_0$, we write the binary representation of $k$ again as
\begin{equation*}
k=\sum_{j=1}^{v}2^{c_j-1},    
\end{equation*}
where $v\in\mathbb{N}_0$ and $c_1>c_2>\cdots>c_v\ge 1$. We then denote $\kappa = \{c_1,\ldots,c_v\}$ as the positions of nonzero bits in the binary expansion of $k$. Henceforth, we identify $k$ with $\kappa$, as each uniquely determines the other. In this framework, $|\kappa|=v$ equals the number of nonzero bits of $k$. For an integer $1\le q\le v$, we denote 
\begin{equation*}
    \lceil \kappa\rceil_q=c_q,\quad \lceil \kappa\rceil_{1{:}q}=\{c_1,\ldots,c_q\}.
\end{equation*}
For $q\ge v$ or $q=0$, we set $\lceil \kappa\rceil_q=0$ by convention.

With this notation, the following lemma provides the $L^1$-version of \cite[Lemma~4]{pan2026dimension}.
\begin{lemma}\label{lemma_partial_L1_sum}
    Let $\alpha\in\mathbb{N}_0$ and $\lambda\in(0,1]$. Suppose $V^v_\lambda(f_u^{(\bm{\alpha}_u)})<\infty$ for nonempty $u\subseteq 1:s, \bm{\alpha}_u=(\alpha_j,j\in u)\in \mathbb{Z}_{\le\alpha}^{|u|}$ and a subset $v\subseteq \{j:\alpha_j=\alpha\}.$ Then for $\bm{k}_u=(k_j,j\in u)\in\mathbb{N}^{|u|}$ satisfying $|\kappa_j|= \alpha+1$ for $j\in v$ and $|\kappa_j|=\alpha_j$ for $j\in u\setminus v$,
    \begin{equation*}
        \sum_{\bm{k}'\in B(\bm{k}_u,v)}|\widehat{f}_u(\bm{k}')|\le 2^{|u|}V^v_\lambda(f_u^{(\bm{\alpha}_u)})\prod_{j\in u}2^{(\frac{3}{2}-\lambda)\lceil \kappa_j\rceil_{\alpha+1}}\prod_{\ell\in\kappa_j}2^{-\ell-1},
    \end{equation*}
    where 
    \begin{equation}
        B(\bm{k}_u,v)=\{\bm{k}'\in\mathbb{N}_{*}^{s}:s(\bm{k}')=u,\lceil \kappa'_j \rceil_{1:(\alpha+1)}=\kappa_j,\forall j\in v,\kappa_j'=\kappa_j,\forall j\in u\setminus v \}.
    \end{equation}
\end{lemma}
\begin{proof}
    According to \cite[Lemma 4]{pan2026dimension}, we have
    \begin{equation*}
        \sum_{\bm{k}'\in B(\bm{k}_u,v)}|\widehat{f}_u(\bm{k}')|^2\le 4^{|u|}\left(V^v_\lambda(f_u^{(\bm{\alpha}_u)})\right)^2\prod_{j\in u}4^{(1-\lambda)\lceil \kappa_j\rceil_{\alpha+1}}\prod_{\ell\in\kappa_j}4^{-\ell-1}.
    \end{equation*}
    Note that 
    \begin{equation*}
        |B(\bm{k}_u,v)|=\prod_{j\in v}2^{\lceil \kappa_j\rceil_{\alpha+1}-1}\le \prod_{j\in u}2^{\lceil \kappa_j\rceil_{\alpha+1}}.
    \end{equation*}
    By Cauchy–Schwarz inequality,
    \begin{align*}
       \left( \sum_{\bm{k}'\in B(\bm{k}_u,v)}|\widehat{f}_u(\bm{k}')|\right)^2 &\le |B(\bm{k}_u,v)|\sum_{\bm{k}'\in B(\bm{k}_u,v)}|\widehat{f}_u(\bm{k}')|^2\\
       &\le \prod_{j\in u}2^{\lceil \kappa_j\rceil_{\alpha+1}}\sum_{\bm{k}'\in B(\bm{k}_u,v)}|\widehat{f}_u(\bm{k}')|^2\\
       &\le 2^{2|u|}\left(V^v_\lambda(f_u^{(\bm{\alpha}_u)})\right)^2\prod_{j\in u}2^{(3-2\lambda)\lceil \kappa_j\rceil_{\alpha+1}}\prod_{\ell\in\kappa_j}2^{-2\ell-2}.\qedhere
    \end{align*}
\end{proof}

The $L^1$-version of \cite[Lemma 5]{pan2026dimension} follows from Lemma \ref{lemma_partial_L1_sum} and the proof of \cite[Lemma 5]{pan2026dimension}.

\begin{lemma}\label{lemma_L1_Ku}
    Suppose $\Vert f\Vert_{u,\alpha,\lambda}<\infty$ for some $\alpha\in\mathbb{N}_0, \lambda\in(0,1]$ and nonempty $u\subseteq 1:s$. Then for $T\in\mathbb{R}$ and $\eta\in(0,1-1/(2\alpha+2\lambda))$,
    \begin{equation*}
        \sum_{\bm{k}\in\mathbb{N}_*^s\setminus K_u(T)}|\widehat{f}_u(\bm{k})|\le \frac{\Vert f\Vert_{u,\alpha,\lambda}}{2^{\eta T}}D_{\alpha,\lambda,\eta}^{|u|},
    \end{equation*}
    where $D_{\alpha,\lambda,\eta}$ is defined in (\ref{constant_D}).
\end{lemma}

\begin{proof}
According to (\ref{wal_f_u}), $\widehat{f}_u(\bm{k})=0$ if $s(\bm{k})\ne u$. Following \cite[Section 4]{pan2026dimension}, we note that 
    \begin{equation*}
        K_u(T)=\bigcup_{v\subseteq u}\bigcup_{\bm{k}_u\in \mathcal{K}_{u,v}(T)}B(\bm{k}_u,v),
    \end{equation*}
    where 
    \begin{equation*}
        \mathcal{K}_{u,v}(T) = \Biggl\{ (k_j, j \in u) \in \mathbb{N}^{|u|} \Biggm| \begin{aligned} &|\kappa_j| = \alpha+1 \ \forall j \in v, \, |\kappa_j| \leqslant \alpha \ \forall j \in u \setminus v, \\ &(\lambda-1) \sum_{j \in u} \lceil \kappa_j \rceil_{\alpha+1} + \sum_{j \in u} \|\kappa_j\| \leqslant T \end{aligned} \Biggr\}.
    \end{equation*}
    Let 
    \begin{equation*}
        \mathbb{N}^{u}_{\alpha}:=\{(k_j,j\in u)\in\mathbb{N}^{|u|}:|\kappa_j|\le \alpha+1,\forall j\in u \}.
    \end{equation*}
    Then 
    \begin{equation*}
        \bigcup_{\substack{\bm{k}_u\in\mathbb{N}_\alpha^u\\ v:=\{j:|\kappa_j|=\alpha+1\} }}B(\bm{k}_u,v) = \{\bm{k}'\in\mathbb{N}_*^s:\bm{s}(\bm{k}')=u\}.
    \end{equation*}
    By Lemma \ref{lemma_partial_L1_sum} and equation (\ref{f_u_a_l}), we have
    \begin{align}
        &\sum_{\bm{k}\in\mathbb{N}_*^s\setminus K_u(T)}|\widehat{f}_u(\bm{k})|\nonumber\\
        &=\sum_{\substack{\bm{k}_u\in\mathbb{N}_\alpha^u\\ v:=\{j:|\kappa_j|=\alpha+1\} }} \sum_{\bm{k}'\in B(\bm{k}_u,v)}|\widehat{f}_u(\bm{k}')| - \sum_{v\subseteq u}\sum_{\bm{k}_u\in \mathcal{K}_{u,v}(T)}\sum_{\bm{k}'\in B(\bm{k}_u,v)}|\widehat{f}_u(\bm{k}')|
        \nonumber\\
        &\le   
        2^{|u|}\Vert f\Vert_{u,\alpha,\lambda}\sum_{\bm{k}_u\in\mathbb{N}_\alpha^u}\bm{1}\left\{  \bm{k}_u\notin\bigcup_{v\subseteq u}\mathcal{K}_{u,v}(T) \right\}\prod_{j\in u}2^{(\frac{3}{2}-\lambda)\lceil \kappa_j\rceil_{\alpha+1}}\prod_{\ell\in\kappa_j}2^{-\ell-1}.\label{sum_up}
    \end{align}
    Because $\bm{k}_u\notin\bigcup_{v\subseteq u}\mathcal{K}_{u,v}(T)$ implies
    \begin{equation*}
        (\lambda-1)\sum_{j\in u}\lceil \kappa_j\rceil_{\alpha+1} + \sum_{j\in u}\Vert\kappa_j\Vert>T,
    \end{equation*}
    we have for any $ \eta\in(0,1-1/(2\alpha+2\lambda))$ and $\rho_\eta=2^{\eta-1}$,
    \begin{align*}
        &\sum_{\bm{k}_u\in\mathbb{N}_\alpha^u}\bm{1}\left\{  \bm{k}_u\notin\bigcup_{v\subseteq u}\mathcal{K}_{u,v}(T) \right\}\prod_{j\in u}2^{(\frac{3}{2}-\lambda)\lceil \kappa_j\rceil_{\alpha+1}}\prod_{\ell\in\kappa_j}2^{-\ell-1}\\
        &\le \sum_{\bm{k}_u\in\mathbb{N}_\alpha^u} 2^{-\eta T +\eta(\lambda-1)\sum_{j\in u}\lceil \kappa_j\rceil_{\alpha+1} + \eta\sum_{j\in u}\Vert\kappa_j\Vert}\prod_{j\in u}2^{(\frac{3}{2}-\lambda)\lceil \kappa_j\rceil_{\alpha+1}}\prod_{\ell\in\kappa_j}2^{-\ell-1}\\
        &=2^{-\eta T}\left(\sum_{\bm{k}_u\in\mathbb{N}_\alpha^u}\prod_{j\in u}2^{(\eta-1)(\lambda-1 )\lceil \kappa_j\rceil_{\alpha+1}+\lceil \kappa_j\rceil_{\alpha+1}/2}\prod_{l\in\kappa_j}2^{(\eta-1)\ell-1} \right)\\
        &=2^{-\eta T}\left(\sum_{\alpha'=0}^{\alpha}2^{-\alpha'}\sum_{\kappa\in \mathcal{S}_{\alpha'}}\prod_{\ell\in\kappa}\rho_\eta^\ell + 2^{-\alpha-1}\sum_{\kappa\in \mathcal{S}_{\alpha+1}}\rho_\eta^{\left(\lambda-1+1/(2\eta-2)\right)\lceil \kappa_j\rceil_{\alpha+1} } \prod_{\ell\in\kappa}\rho_\eta^\ell\right)^{|u|},
    \end{align*}
    where $\mathcal{S}_{\alpha'}:=\{\kappa\subseteq \mathbb{N}:|\kappa|=\alpha'\}$. According to \cite[equation (20)]{pan2026dimension}, we have
    \begin{equation*}
        \sum_{\kappa\in \mathcal{S}_{\alpha'}}\prod_{\ell\in\kappa}\rho_\eta^\ell\le \frac{1}{(\alpha')!(\rho_\eta^{-1}-1)^{\alpha'}}.
    \end{equation*}
    For $\kappa\in S_{\alpha+1}$, we let $\kappa^+=\lceil \kappa\rceil_{1:\alpha}$ and $\ell'=\lceil \kappa\rceil_{\alpha+1}$. By the one-to-one correspondence between $\kappa^+$ and $\kappa^+-\ell'\in \mathcal{S}_{\alpha}$,
    \begin{align*}
        \sum_{\kappa\in \mathcal{S}_{\alpha+1}}\rho_\eta^{\left(\lambda-1+1/(2\eta-2)\right)\lceil \kappa_j\rceil_{\alpha+1} } \prod_{\ell\in\kappa}\rho_\eta^\ell&\le \sum_{\ell'=1}^{\infty}\rho_\eta^{(\alpha+\lambda+1/(2\eta-2))\ell'}\sum_{\kappa^+-\ell'\in \mathcal{S}_{\alpha}}\prod_{\ell\in \kappa^+-\ell'}\rho_\eta^{\ell}\\
        &\le \sum_{\ell'=1}^{\infty}\rho_\eta^{(\alpha+\lambda+1/(2\eta-2))\ell'}\frac{1}{\alpha!(\rho_{\eta}^{-1}-1)^{\alpha}}\\
        &=\frac{1}{\alpha!(\rho_{\eta}^{-\alpha-\lambda-1/(2\eta-2)}-1)(\rho_{\eta}^{-1}-1)^{\alpha}}.
    \end{align*}
    Our conclusion follows once we put the above bounds into inequality (\ref{sum_up}).
\end{proof}

Combining Lemma \ref{lemma_partial_L1_sum} and Lemma \ref{lemma_L1_Ku}, we obtain Lemma \ref{lemma_psi}.
\begin{proof}[Proof of Lemma \ref{lemma_psi}]
    According to Lemma \ref{lemma_L1_Ku}, for $\eta\in (0,1-1/(2\alpha+2\lambda))$,
    \begin{align*}
        \sum_{\bm{k}\in\mathbb{N}_0^s\setminus K_{m,\bm{\gamma}}}|\widehat{f}(\bm{k})|&\le \sum_{u\subseteq 1:s}\frac{\Vert f\Vert_{u,\alpha,\lambda}}{2^{\eta T_{u,m,\bm{\gamma}}}}D_{\alpha,\lambda,\eta}^{|u|}\\
        &= \frac{\Gamma_{\alpha,\lambda,\beta,\bm{\gamma}}(m)^{\eta(\alpha+\lambda)}\Vert f\Vert_{s,\alpha,\lambda} }{\delta^{\eta(\alpha+\lambda)}2^{\eta(\alpha+\lambda)m}}\sum_{u\subseteq 1:s}\gamma_{u}^{1-\frac{\eta(\alpha+\lambda)}{\alpha+\lambda+1/2}}\beta^{\eta(\alpha+\lambda)|u|} D_{\alpha,\lambda,\eta}^{|u|}\\
        &=\Theta_{\alpha,\lambda,\beta,\bm{\gamma},\delta,\eta}\frac{\Gamma_{\alpha,\lambda,\beta,\bm{\gamma}}(m)^{\eta(\alpha+\lambda)}}{2^{\eta(\alpha+\lambda)m}}\Vert f\Vert_{s,\alpha,\lambda},
    \end{align*}
    where in the second equality we substituted the value of $T_{u,m,\bm{\gamma}}$ and used the definition of $\gamma_u$ in (\ref{gamma_u}).
\end{proof}

\section{Proof of Theorem \ref{thm_main}}\label{appendix_main}
For theoretical analysis, the elements of $K$ are ordered as $\{\bm{k}_j'\}_{j=1}^{|K|}$ such that $|\widehat{f}(\bm{k}_1')| \ge |\widehat{f}(\bm{k}_2')| \ge \cdots$. Accordingly, the reference set is defined as $H_N' := \{\bm{k}_1', \ldots, \bm{k}_N'\}$ if $|K| \ge N$, and $H_N' := K$ otherwise. Similar to \cite[inequality ($\clubsuit$)]{pan2025universal}, we further decompose the error terms in \eqref{decomp_1} as
\begin{align}
    &\Vert A^{E,\epsilon}_m(f)-f \Vert_{L^2}^2\nonumber\\
    &=\sum_{\bm{k}\in H^{E,\epsilon}_N}|\widehat{f}^{E,\epsilon}_{N,\med}(\bm{k})-\widehat{f}(\bm{k})|^2 + \sum_{\bm{k}\in \mathbb{N}_0^s\setminus H^{E,\epsilon}_N}|\widehat{f}(\bm{k})|^2\nonumber\\
    &\le \mathrm{I}+ \mathrm{II}+\mathrm{III},\label{decomp}
\end{align}
where
\begin{align*}
  \mathrm{I} =   \sum_{\bm{k}\in H^{E,\epsilon}_N}|\widehat{f}^{E,\epsilon}_{N,\med}(\bm{k})-\widehat{f}(\bm{k})|^2,\quad \mathrm{II} =\sum_{\bm{k}\in\mathbb{N}_0^s\setminus H_N'}|\widehat{f}(\bm{k})|^2,\quad  \mathrm{III} = \sum_{\bm{k}\in H_N'\setminus H^{E,\epsilon}_N}|\widehat{f}(\bm{k})|^2.
\end{align*}
We analyze these three terms individually; their probabilistic bounds are given in Theorems \ref{thm_term1}, \ref{thm_term2}, and \ref{thm_term3}. Theorem \ref{thm_main} then follows by combining these results.

Before delving into the analysis of each term, we recall \cite[Proposition 2.2]{kunsch2019solvable}, which characterizes how the median-of-means technique suppresses the failure probability of obtaining inaccurate estimates.
\begin{lemma}\label{lemma_median_trick}
    Let $R$ be an odd natural number, and let $X_1,\ldots,X_R$ be $R$ independent and identically
distributed real random variables. Suppose that for some interval $I\subseteq R$ and $p\in(0,1/2)$ we have
\begin{equation*}
    \operatorname{Pr}\left(X_1\notin I \right)\le p.
\end{equation*}
Then 
\begin{equation*}
    \operatorname{Pr}\left(\underset{r\in 1:R}{\operatorname{med}}X_r\notin I \right)\le \frac{1}{2}\left(4p(1-p)\right)^{R/2}.
\end{equation*}
\end{lemma}

The next lemma shows that the median does not increase the error, which will be used in our analysis of $\widehat{f}^{E,\epsilon}_{N,\med}$ and $H^{E,\epsilon}_N$.

\begin{lemma}\label{lemma_median_perserve}
    Let $\epsilon>0$, $R$ be an odd natural number, and $x_1,\ldots,x_R$ be $R$ real numbers. Let $y_1,\ldots,y_R$ be $R$ real numbers satisfying $|x_r-y_r|\le\epsilon$ for all $1\le r\le R$. Then
    \begin{equation*}
        |\underset{r\in 1:R}{\operatorname{med}}x_r-\underset{r\in 1:R}{\operatorname{med}}y_r|\le \epsilon.
    \end{equation*}
\end{lemma}
\begin{proof}
    Denote $z_r=x_r+\epsilon$ for $1\le r\le R$. Then $y_r\le z_r$. So 
    \begin{equation*}
        \underset{r\in 1:R}{\operatorname{med}}y_r\le \underset{r\in 1:R}{\operatorname{med}}z_r=\underset{r\in 1:R}{\operatorname{med}}x_r+\epsilon.
    \end{equation*}
    Similarly, we have 
    \[\underset{r\in 1:R}{\operatorname{med}}y_r\ge \underset{r\in 1:R}{\operatorname{med}}x_r-\epsilon.\]
    Hence,
    \[    |\underset{r\in 1:R}{\operatorname{med}}x_r-\underset{r\in 1:R}{\operatorname{med}}y_r|\le \epsilon.\qedhere\] 
\end{proof}

\subsection{Analysis of the term $\mathrm{I}$ in \eqref{decomp}}
To begin with, we recall a technical lemma from \cite[Lemma 1]{pan2025automatic}.
\begin{lemma}\label{lemma_set_c_r}
    For $1\le r\le R$, denote 
    \begin{equation}\label{set_C_r}
        \bm{C}_r^\perp:=\left\{ \el\in\mathbb{N}_0^s:\sum_{j=1}^{s}\Vec{\ell}_j^\top C_{jr}=\bm{0}\mod2\right\}.
    \end{equation}
    Then for $\el\ne \bm{0}$, we have
    \[\operatorname{Pr}(\el\in\bm{C}_r^\perp )=2^{-m}=\frac{1}{N}.\]
\end{lemma}

Combining this lemma and Lemma \ref{lemma_median_trick}, we obtain the probabilistic bound for $\widehat{f}_{N,\med}$.
\begin{lemma}\label{median_first}
    For $\bm{k}\in K$ and $\delta\in (0,\frac{1}{4})$, 
    \begin{equation*}
        \operatorname{Pr}\left(|\widehat{f}_{N,\med}(\bm{k})-\widehat{f}(\bm{k})|^2>\frac{\Lambda_2}{\delta N}\right)\le \frac{1}{2}(8\delta(1-2\delta))^{R/2},
    \end{equation*}
    where $\Lambda_2$ is defined in (\ref{Lambda_m}).
\end{lemma}
\begin{proof}
    For each $r\in 1:R$, by Lemma \ref{lemma_error_g}, together with the pairwise independence of $W(\el)$,  and the independence between $\bm{C}_r$ and $\bm{D}_r$, we have
    \begin{equation}\label{E_D_r}
        \mathbb{E}_{\bm{D}_r}\left[|\widehat{f}_{N,r}(\bm{k})-\widehat{f}(\bm{k})|^2\right]=\sum_{\el\in \bm{C}_r^\perp\setminus\{\bm{0}\}}|\widehat{f}(\bm{k}\oplus\el)|^2,
    \end{equation}
    where $\bm{C}_r^\perp$ is defined in (\ref{set_C_r}). Conditioned on the event $(\bm{k}\oplus \bm{C}_r^\perp)\cap K_{m,\bm{\gamma}}\subseteq\{\bm{k}\}$,
    \begin{align}
        &\mathbb{E}_{\bm{C}_r}\left[\bm{1}\left\{(\bm{k}\oplus \bm{C}_r^\perp)\cap K_{m,\bm{\gamma}}\subseteq\{\bm{k}\} \right\}\sum_{\el\in \bm{C}_r^\perp\setminus\{\bm{0}\}}|\widehat{f}(\bm{k}\oplus\el)|^2 \right]\nonumber\\
        &=\mathbb{E}_{\bm{C}_r}\left[\bm{1}\left\{(\bm{k}\oplus \bm{C}_r^\perp)\cap K_{m,\bm{\gamma}}\subseteq\{\bm{k}\} \right\}\sum_{\bm{k}'\in \mathbb{N}_0^s\setminus\{\bm{k}\} }|\widehat{f}(\bm{k}')|^2\bm{1}\left\{\bm{k}'\oplus\bm{k}\in \bm{C}_r^\perp \right\}  \right]\nonumber\\
        &\le \sum_{\bm{k}'\in\mathbb{N}_0^s\setminus (K_{m,\bm{\gamma}}\cup\{\bm{k}\} )}\operatorname{Pr}(\bm{k}'\oplus \bm{k}\in \bm{C}_r^\perp)|\widehat{f}(\bm{k}')|^2\le\frac{\Lambda_2}{N},\label{E_C_r_L2}
    \end{align}
    where in the first equality we set $\bm{k}'=\bm{k}\oplus\el \in \mathbb{N}_0^s\setminus\{\bm{k}\}$, in the last inequality we used Lemma \ref{lemma_set_c_r}, and in the second inequality we noted that if $\bm{k}'\in K_{m,\bm{\gamma}}$ and $\el=\bm{k}'\oplus\bm{k}\in \bm{C}_r^\perp\setminus\{\bm{0}\}$, then 
    \[\bm{k}\ne\bm{k}'=\bm{k}\oplus\el\in (\bm{k}\oplus \bm{C}_r^\perp)\cap K_{m,\bm{\gamma}} .\]
    By Lemma \ref{card} and Lemma \ref{lemma_set_c_r},
    \begin{align}
        \operatorname{Pr}\left((\bm{k}\oplus \bm{C}_r^\perp)\cap K_{m,\bm{\gamma}}\not\subseteq\{\bm{k}\}\right)&=\operatorname{Pr}\left(\bigcup_{\bm{k}'\in K_{m,\bm{\gamma}}}\{\bm{k}'\oplus \bm{k}\in \bm{C}_r^\perp\setminus\{\bm{0}\} \}  \right)\nonumber\\
        &\le \sum_{\bm{k}'\in K_{m,\bm{\gamma}}}\operatorname{Pr}\left(\bm{k}'\oplus \bm{k}\in \bm{C}_r^\perp\setminus\{\bm{0}\} \right)\nonumber\\
        &\le \frac{|K_{m,\bm{\gamma}}|}{N}\le \delta.\label{Pr_set}
    \end{align}
    Hence, by Markov’s inequality
    \begin{align}
        &\operatorname{Pr}\left( |\widehat{f}_{N,r}(\bm{k})-\widehat{f}(\bm{k})|^2>\frac{\Lambda_2}{\delta N}\right)\nonumber\\
        &\le \operatorname{Pr}\left((\bm{k}\oplus \bm{C}_r^\perp)\cap K_{m,\bm{\gamma}}\not\subseteq\{\bm{k}\}\right)\nonumber\\
        &\quad+\operatorname{Pr}\left(\frac{\delta N}{\Lambda_2}|\widehat{f}_{N,r}(\bm{k})-\widehat{f}(\bm{k})|^2 \bm{1}\left\{(\bm{k}\oplus \bm{C}_r^\perp)\cap K_{m,\bm{\gamma}}\subseteq\{\bm{k}\} \right\}> 1      \right)\nonumber\\
        &\le \operatorname{Pr}\left((\bm{k}\oplus \bm{C}_r^\perp)\cap K_{m,\bm{\gamma}}\not\subseteq\{\bm{k}\}\right)\nonumber\\
        &\quad+ \frac{\delta N}{\Lambda_2}\mathbb{E}_{\bm{C}_r}\left[\bm{1}\left\{(\bm{k}\oplus \bm{C}_r^\perp)\cap K_{m,\bm{\gamma}}\subseteq\{\bm{k}\} \right\}\sum_{\el\in \bm{C}_r^\perp\setminus\{\bm{0}\}}|\widehat{f}(\bm{k}\oplus\el)|^2 \right]\nonumber\\
        &\le 2\delta,\label{Pr_median}
    \end{align}
    where the second inequality also used (\ref{E_D_r}).
    The conclusion follows from Lemma~\ref{lemma_median_trick} and the observation that $\widehat{f}_{N,r}(\bm{k})\in\mathbb{R}$.
\end{proof}

We now show a bound on the first term $\mathrm{I}$ in (\ref{decomp}).
\begin{theorem}\label{thm_term1}
    Let $\delta\in(0,\frac{1}{4}),  N=2^m$ and let $R$ be an odd natural number satisfying
    \begin{equation*}
        \delta_1:=\frac{N}{2}(8\delta(1-2\delta))^{R/2}<1.
    \end{equation*}
    Assume $K\supseteq K_{m,\bm{\gamma}}$ and the condition (\ref{condition_E}) holds.
    Then, for $\theta\in(0,1)$ and any $f$ with $\Vert f\Vert_{s,\alpha,\lambda}<\infty$, with probability at least $1-\delta_1$ we have
    \begin{equation*}
        \sum_{\bm{k}\in H^{E,\epsilon}_N}|\widehat{f}^{E,\epsilon}_{N,\med}(\bm{k})-\widehat{f}(\bm{k})|^2\le \Upsilon_{\alpha,\lambda,\beta, \bm{\gamma},\delta,\theta}\frac{2\Gamma_{\alpha,\lambda,\beta,\bm{\gamma}}(m)^{2\theta(\alpha+\lambda)}}{\delta N^{2\theta(\alpha+\lambda)}}\Vert f\Vert_{s,\alpha,\lambda}^2+2N\epsilon_1^2,
    \end{equation*}
    where $\epsilon_1$, $\Gamma_{\alpha,\lambda,\beta,\bm{\gamma}}(m)$ and $\Upsilon_{\alpha,\lambda,\beta, \bm{\gamma},\delta,\theta}$ are defined in (\ref{constant_ep1}), (\ref{Gamma(m)}) and (\ref{constant_Gamma}), respectively.
\end{theorem}
\begin{proof}
    According to (\ref{assumption_ep}) and Lemma \ref{lemma_median_perserve}, for any $\bm{k}\in K$, we have
    \[|\widehat{f}_{N,\med}(\bm{k})-\widehat{f}^{E,\epsilon}_{N,\med}(\bm{k})|\le \epsilon_1.\]
    By Lemma \ref{median_first}, 
    \[
    \operatorname{Pr}\left(\exists\bm{k}\in H^{E,\epsilon}_N,\text{ s.t. }|\widehat{f}_{N,\med}(\bm{k})-\widehat{f}(\bm{k})|^2>\frac{\Lambda_2}{\delta N} \right)\le \frac{1}{2}(8\delta(1-2\delta))^{R/2}|H^{E,\epsilon}_N|\le\delta_1.
    \]
    Hence, with probability at least $1-\delta_1$ we have $|\widehat{f}_{N,\med}(\bm{k})-\widehat{f}(\bm{k})|^2\le\Lambda_2/(\delta N) $ for all $\bm{k}\in H^{E,\epsilon}_N$. It then follows from Lemma \ref{lemma_lambda} and Cauchy–Schwarz inequality that
    \begin{align*}
        &\sum_{\bm{k}\in H^{E,\epsilon}_N}|\widehat{f}^{E,\epsilon}_{N,\med}(\bm{k})-\widehat{f}(\bm{k})|^2\\
        &\le 2\sum_{\bm{k}\in H^{E,\epsilon}_N}|\widehat{f}_{N,\med}(\bm{k})-\widehat{f}(\bm{k})|^2+2\sum_{\bm{k}\in H^{E,\epsilon}_N}|\widehat{f}_{N,\med}(\bm{k})-\widehat{f}^{E,\epsilon}_{N,\med}(\bm{k})|^2\\
        &\le \frac{2\Lambda_2}{\delta}+2N\epsilon_1^2\\
        &\le \Upsilon_{\alpha,\lambda,\beta, \bm{\gamma},\delta,\theta}\frac{2\Gamma_{\alpha,\lambda,\beta,\bm{\gamma}}(m)^{2\theta(\alpha+\lambda)}}{\delta N^{2\theta(\alpha+\lambda)}}\Vert f\Vert_{s,\alpha,\lambda}^2+2N\epsilon_1^2. \qedhere
    \end{align*}
\end{proof}

\subsection{Analysis of the term $\mathrm{II}$ in \eqref{decomp}} The upper bound for the second term $\mathrm{II}$ in (\ref{decomp}) can be obtained by definition of $H_N'$.
\begin{theorem}\label{thm_term2}
    Assume $K\supseteq K_{m,\bm{\gamma}}$.
    For $\theta\in(0,1)$ and $f$ with $\Vert f\Vert_{s,\alpha,\lambda}<\infty$, we have 
    \begin{equation*}
        \sum_{\bm{k}\in\mathbb{N}_0^s\setminus H_N' }|\widehat{f}(\bm{k})|^2\le \Upsilon_{\alpha,\lambda,\beta, \bm{\gamma},\delta,\theta}\frac{\Gamma_{\alpha,\lambda,\beta,\bm{\gamma}}(m)^{2\theta(\alpha+\lambda)}}{ N^{2\theta(\alpha+\lambda)}}\Vert f\Vert_{s,\alpha,\lambda}^2,
    \end{equation*}
    where $\Gamma_{\alpha,\lambda,\beta,\bm{\gamma}}(m)$ and $\Upsilon_{\alpha,\lambda,\beta, \bm{\gamma},\delta,\theta}$ are defined in (\ref{Gamma(m)}) and (\ref{constant_Gamma}), respectively.
\end{theorem}
\begin{proof}
If $|K|<N$, then $K_{m,\bm{\gamma}}\subseteq K=H_N'$ and we have
\[ \sum_{\bm{k}\in\mathbb{N}_0^s\setminus H_N' }|\widehat{f}(\bm{k})|^2\le \sum_{\bm{k}\in\mathbb{N}_0^s\setminus K_{m,\bm{\gamma}} }|\widehat{f}(\bm{k})|^2. \]
If $|K|\ge N$, then 
$|K_{m,\bm{\gamma}}|\le \delta N<N=|H_N'|$. Since $K_{m,\bm{\gamma}}\subseteq K$, by the definition of $H_N'$ we also have 
\[ \sum_{\bm{k}\in\mathbb{N}_0^s\setminus H_N' }|\widehat{f}(\bm{k})|^2\le \sum_{\bm{k}\in\mathbb{N}_0^s\setminus K_{m,\bm{\gamma}} }|\widehat{f}(\bm{k})|^2. \]
The conclusion follows from Lemma \ref{lemma_lambda}.
\end{proof}

\subsection{Analysis of the term $\mathrm{III}$ in \eqref{decomp}}
Finally, we consider the third term $\mathrm{III}$ in (\ref{decomp}). Note that the case where $|K|<N$ is trivial, as $H_N'=K=H^{E,\epsilon}_N$ in this scenario. Therefore, the remainder of this section focuses on the case $|K|\ge N$.
Similar to \cite[Section 4.3]{pan2025universal}, we will use the fourth moment of $\widehat{f}_{N,r}(\bm{k})$ to show that
with high probability, all $\bm{k}\in H_N'\setminus H^{E,\epsilon}_N$ corresponds to small values of $|\widehat{f}(\bm{k})|$.

We begin by computing the second and fourth moments of  $\widehat{f}_{N,r}(\bm{k})$ and providing an upper bound for the variance of $|\widehat{f}_{N,r}(\bm{k})|^2$  under the expectation over $\bm{D}_r$.
\begin{lemma}\label{2_4_matrix}
For any $r\in 1:R$ and $\bm{k}\in \mathbb{N}_0^s$, we have
    \begin{align*}
        \mathbb{E}_{\bm{D}_r}\left[|\widehat{f}_{N,r}(\bm{k})|^2\right]&=\sum_{\el\in \bm{C}_r^\perp}|\widehat{f}(\bm{k}\oplus\el)|^2,\\
        \mathbb{E}_{\bm{D}_r}\left[|\widehat{f}_{N,r}(\bm{k})|^4\right]&=\sum_{\el\in \bm{C}_r^\perp}\left|\sum_{\el'\in \bm{C}_r^\perp}\widehat{f}(\bm{k}\oplus\el')\widehat{f}(\bm{k}\oplus\el\oplus\el') \right|^2.
    \end{align*}
\end{lemma}
\begin{proof}
    The proof closely parallels that of \cite[Lemma 4.7]{pan2025universal}, with the only modification being that we invoke Lemma \ref{lemma_error_g} to express $\widehat{f}_{N,r}(\bm{k})$. 
\end{proof}

\begin{lemma}\label{lemma_sigma}
    For any $r\in 1:R$ and $\bm{k}\in \mathbb{N}_0^s$, define
    \begin{align*}
        t(\bm{k}):=\sum_{\el\in \bm{C}_r^\perp\setminus\{\bm{0}\} }|\widehat{f}(\bm{k}\oplus\el)|,\quad S(\bm{k}):=\sum_{\el\in \bm{C}_r^\perp\setminus\{\bm{0}\} }|\widehat{f}(\bm{k}\oplus\el)|^2.
    \end{align*}
    Then we have
    \begin{equation*}
        \sigma^2(\bm{k}):=\mathbb{E}_{\bm{D}_r}\left[|\widehat{f}_{N,r}(\bm{k})|^4\right]-\left(\mathbb{E}_{\bm{D}_r}\left[|\widehat{f}_{N,r}(\bm{k})|^2\right]\right)^2\le 8|\widehat{f}(\bm{k})|^2 S(\bm{k})+2t(\bm{k})^2 S(\bm{k}).
    \end{equation*}
\end{lemma}

\begin{proof}
    For $\el = \bm{0}$, we have $\sum_{\el'\in \bm{C}_r^\perp}\widehat{f}(\bm{k}\oplus\el')\widehat{f}(\bm{k}\oplus\el\oplus\el') = \sum_{\el'\in \bm{C}_r^\perp}|\widehat{f}(\bm{k}\oplus\el')|^2$. Hence, by Lemma \ref{2_4_matrix}, the variance is given by
    \begin{equation*}
        \sigma^2(\bm{k})=\sum_{\el\in \bm{C}_r^\perp\setminus\{\bm{0}\} }\left|\sum_{\el'\in \bm{C}_r^\perp}\widehat{f}(\bm{k}\oplus\el')\widehat{f}(\bm{k}\oplus\el\oplus\el') \right|^2.
    \end{equation*}
    For any $\el\in \bm{C}_r^\perp\setminus\{\bm{0}\}$, by Cauchy–Schwarz inequality,
    \begin{align*}
        &\left|\sum_{\el'\in \bm{C}_r^\perp}\widehat{f}(\bm{k}\oplus\el')\widehat{f}(\bm{k}\oplus\el\oplus\el') \right|^2\\
        &=\left|2\widehat{f}(\bm{k})\widehat{f}(\bm{k}\oplus\el) + \sum_{\el'\in \bm{C}_r^\perp\setminus\{\bm{0},\el\} }\widehat{f}(\bm{k}\oplus\el')\widehat{f}(\bm{k}\oplus\el\oplus\el')\right|\\
        &\le 8|\widehat{f}(\bm{k})|^2|\widehat{f}(\bm{k}\oplus\el)|^2 + 2\left|\sum_{\el'\in \bm{C}_r^\perp\setminus\{\bm{0},\el\} }\widehat{f}(\bm{k}\oplus\el')\widehat{f}(\bm{k}\oplus\el\oplus\el') \right|^2.
    \end{align*}
    A further application of Cauchy–Schwarz inequality to the second term yields:
    \begin{align*}
        &\left|\sum_{\el'\in \bm{C}_r^\perp\setminus\{\bm{0},\el\} }\widehat{f}(\bm{k}\oplus\el')\widehat{f}(\bm{k}\oplus\el\oplus\el') \right|^2\\
        &\le \left(\sum_{\el'\in \bm{C}_r^\perp\setminus\{\bm{0},\el\} }|\widehat{f}(\bm{k}\oplus\el')| \right)\left(\sum_{\el'\in \bm{C}_r^\perp\setminus\{\bm{0},\el\} }|\widehat{f}(\bm{k}\oplus\el')||\widehat{f}(\bm{k}\oplus\el\oplus\el')|^2 \right)\\
        &\le t(\bm{k})\left(\sum_{\el'\in \bm{C}_r^\perp\setminus\{\bm{0},\el\} }|\widehat{f}(\bm{k}\oplus\el')||\widehat{f}(\bm{k}\oplus\el\oplus\el')|^2 \right).
    \end{align*}
    Thus,
    \begin{align*}
        \sigma^2(\bm{k})
        &\le2t(\bm{k})\left(\sum_{\el\in \bm{C}_r^\perp\setminus\{\bm{0}\} }\sum_{\el'\in \bm{C}_r^\perp\setminus\{\bm{0},\el\} }|\widehat{f}(\bm{k}\oplus\el')||\widehat{f}(\bm{k}\oplus\el\oplus\el')|^2 \right)\\
        &\quad +8|\widehat{f}(\bm{k})|^2\sum_{\el\in \bm{C}_r^\perp\setminus\{\bm{0}\} } |\widehat{f}(\bm{k}\oplus\el)|^2\\
        &= 8|\widehat{f}(\bm{k})|^2 S(\bm{k}) + 2t(\bm{k})\left(\sum_{\el'\in \bm{C}_r^\perp\setminus\{\bm{0}\}}\sum_{\bm{L}\in \bm{C}_r^\perp\setminus\{\bm{0},\el'\}}|\widehat{f}(\bm{k}\oplus\el')||\widehat{f}(\bm{k}\oplus\bm{L})|^2\right)\\
        &\le 8|\widehat{f}(\bm{k})|^2 S(\bm{k}) + 2t(\bm{k})\left(\sum_{\el'\in \bm{C}_r^\perp\setminus\{\bm{0}\}}\sum_{\bm{L}\in \bm{C}_r^\perp\setminus\{\bm{0}\}}|\widehat{f}(\bm{k}\oplus\el')||\widehat{f}(\bm{k}\oplus\bm{L})|^2\right)\\
        &=8|\widehat{f}(\bm{k})|^2 S(\bm{k})+2t(\bm{k})^2 S(\bm{k}),
    \end{align*}
    where in the second equality we interchange the order of summation and write $\bm{L}=\el\oplus\el'$.
\end{proof}

\begin{lemma}\label{lemma_s_r}
    For $\bm{k}\in K$, we have
    \begin{align*}
        \operatorname{Pr}\left(S(\bm{k})>\frac{\Lambda_2}{\delta N} \right)\le 2\delta\quad\text{and}\quad\operatorname{Pr}\left(t(\bm{k})>\frac{\Lambda_1}{\delta N} \right)\le 2\delta,
    \end{align*}
    where $\Lambda_2$ and $\Lambda_1$ are defined in (\ref{Lambda_m}) and (\ref{Psi_m}), respectively.
\end{lemma}

\begin{proof}
    Similar to (\ref{Pr_median}) in the proof of Lemma \ref{median_first}, by applying Markov’s inequality we have
    \begin{align*}
        &\operatorname{Pr}\left(S(\bm{k})>\frac{\Lambda_2}{\delta N} \right)\\
        &\le \operatorname{Pr}\left((\bm{k}\oplus \bm{C}_r^\perp)\cap K_{m,\bm{\gamma}}\not\subseteq\{\bm{k}\}\right)+ \operatorname{Pr}\left( \frac{\delta N}{\Lambda_2}S(\bm{k}) \bm{1}\left\{(\bm{k}\oplus \bm{C}_r^\perp)\cap K_{m,\bm{\gamma}}\subseteq\{\bm{k}\} \right\}> 1 \right)\\
        &\le \delta + \frac{\delta N}{\Lambda_2}\mathbb{E}_{\bm{C}_r}\left[\bm{1}\left\{(\bm{k}\oplus \bm{C}_r^\perp)\cap K_{m,\bm{\gamma}}\subseteq\{\bm{k}\} \right\}\sum_{\el\in \bm{C}_r^\perp\setminus\{\bm{0}\}}|\widehat{f}(\bm{k}\oplus\el)|^2 \right]\\
        &\le 2\delta.
    \end{align*}
    The bound for $t(\bm{k})$ follows in a similar way by replacing $\sum_{\el\in \bm{C}_r^\perp\setminus\{\bm{0}\}}|\widehat{f}(\bm{k}\oplus\el)|^2$ with $\sum_{\el\in \bm{C}_r^\perp\setminus\{\bm{0}\}}|\widehat{f}(\bm{k}\oplus\el)|$.
\end{proof}

\begin{lemma}
    For $f$ with $\Vert f\Vert_{s,\alpha,\lambda,\bm{\gamma}}<\infty$, we have
    \begin{equation*}
        |\widehat{f}(\bm{k}_N')|^2\le \frac{\Lambda_2}{(1-\delta)N},
    \end{equation*}
    where $\Lambda_2$ is defined in (\ref{Lambda_m}).
\end{lemma}

\begin{proof}
    Since $|K_{m,\bm{\gamma}}|\le \delta N$, we have $|H_N'\setminus K_{m,\bm{\gamma}}|\ge (1-\delta)N$. Thus,
    \begin{equation*}
        \Lambda_2:=\sum_{\bm{k}\in\mathbb{N}_0^s\setminus K_{m,\bm{\gamma}}}|\widehat{f}(\bm{k})|^2\ge \sum_{\bm{k}\in H_N'\setminus K_{m,\bm{\gamma}}}|\widehat{f}(\bm{k})|^2\ge (1-\delta)N |\widehat{f}(\bm{k}_N')|^2,
    \end{equation*}
    and the conclusion follows.
\end{proof}

\begin{corollary}\label{cor_K_minus_H}
    For any $\delta\in(0,\frac{1}{4})$ and  $\bm{k}\in K\setminus H_N'$, we have
    \begin{equation*}
        \operatorname{Pr}\left( \underset{r\in 1:R}{\operatorname{med}}|\widehat{f}_{N,r}(\bm{k})|^2>\frac{(2-\delta)\Lambda_2}{\delta(1-\delta)N}\right)\le \frac{1}{2}(8\delta(1-2\delta))^{R/2},
    \end{equation*}
    where $\Lambda_2$ is defined in (\ref{Lambda_m}).
\end{corollary}
\begin{proof}
    For any $\bm{k}\in K\setminus H_N'$, we have
    \begin{equation*}
        |\widehat{f}(\bm{k})|^2\le |\widehat{f}(\bm{k}_N')|^2\le \frac{\Lambda_2}{(1-\delta)N}.
    \end{equation*}
    Similar to (\ref{E_C_r_L2}) in the proof of Lemma \ref{median_first}, for any $r\in 1:R$,
    \begin{align*}
        &\mathbb{E}\left[|\widehat{f}_{N,r}(\bm{k})|^2\bm{1}\left\{(\bm{k}\oplus \bm{C}_r^\perp)\cap K_{m,\bm{\gamma}}\subseteq\{\bm{k}\} \right\}\right]\\
        &=\mathbb{E}_{\bm{C}_r}\left[\bm{1}\left\{(\bm{k}\oplus \bm{C}_r^\perp)\cap K_{m,\bm{\gamma}}\subseteq\{\bm{k}\} \right\}\sum_{\el\in \bm{C}_r^\perp}|\widehat{f}(\bm{k}\oplus\el)|^2 \right]\\
        &\le \frac{\Lambda_2}{N} + |\widehat{f}(\bm{k})|^2\le \frac{(2-\delta)\Lambda_2}{(1-\delta)N}.
    \end{align*}
    Hence, by (\ref{Pr_set}) and Markov’s inequality
    \begin{align*}
        &\operatorname{Pr}\left(|\widehat{f}_{N,r}(\bm{k})|^2>\frac{(2-\delta)\Lambda_2}{\delta(1-\delta)N} \right)\\
        &\le \operatorname{Pr}\left((\bm{k}\oplus \bm{C}_r^\perp)\cap K_{m,\bm{\gamma}}\not\subseteq\{\bm{k}\}\right)\\
        &\quad+\frac{\delta(1-\delta)N}{(2-\delta)\Lambda_2}\mathbb{E}\left[|\widehat{f}_{N,r}(\bm{k})|^2\bm{1}\left\{(\bm{k}\oplus \bm{C}_r^\perp)\cap K_{m,\bm{\gamma}}\subseteq\{\bm{k}\} \right\}\right]\\
        &\le 2\delta.
    \end{align*}
    The conclusion follows from Lemma \ref{lemma_median_trick}.
\end{proof}

\begin{lemma}\label{lemma_with_condition}
    For $\bm{k}\in K, \delta\in (0,\frac{1}{10}), \theta\in(0,1)$, and $f$ with $\Vert f\Vert_{s,\alpha,\lambda}< \infty$, suppose that 
    \begin{equation}\label{condition_hat_f}
        |\widehat{f}(\bm{k})|\ge G_{\alpha,\lambda,\beta,\bm{\gamma},\delta,\theta} \frac{\Gamma_{\alpha,\lambda,\beta,\bm{\gamma}}(m)^{\theta(\alpha+\lambda)}}{N^{\theta(\alpha+\lambda+1/2)}}\Vert f\Vert_{s,\alpha,\lambda},
    \end{equation}
    where
        \begin{equation}\label{constant_G}
        G_{\alpha,\lambda,\beta,\bm{\gamma},\delta,\theta}=\delta^{-1}(1-\delta)^{\frac{1}{4}}\max\left\{8\Upsilon_{\alpha,\lambda,\beta, \bm{\gamma},\delta,\theta}^{\frac{1}{2}},\frac{1}{2}\Theta_{\alpha,\lambda,\beta,\bm{\gamma},\delta,\eta}\right\}
    \end{equation}
    with $\eta=(1-1/(2\alpha+2\lambda))\theta$, and $\Upsilon_{\alpha,\lambda,\beta, \bm{\gamma},\delta,\theta}$ and $\Theta_{\alpha,\lambda,\beta,\bm{\gamma},\delta,\eta}$ are defined in (\ref{constant_Gamma}) and (\ref{constant_Theta}), respectively.
    Then 
    \begin{align*}
        \operatorname{Pr}\left(\underset{r\in 1:R}{\operatorname{med}}|\widehat{f}_{N,r}(\bm{k})|^2\le \frac{|\widehat{f}(\bm{k})|^2}{2} \right)\le \frac{1}{2}(20\delta(1-5\delta))^{R/2}.
    \end{align*}
\end{lemma}
\begin{proof}
    Without loss of generality, we assume $\Vert f\Vert_{s,\alpha,\lambda}\le 1$. Note that for any $r\in 1:R$, we have
    \begin{equation*}
        \mathbb{E}_{\bm{D}_r}\left[|\widehat{f}_{N,r}(\bm{k})|^2\right]=\sum_{\el\in \bm{C}_r^\perp}|\widehat{f}(\bm{k}\oplus\el)|^2\ge |\widehat{f}(\bm{k})|^2.
    \end{equation*}
    Hence, by Cantelli’s inequality,
    \begin{align}
        &\mathbb{E}_{\bm{D}_r}\left[\bm{1}\left\{|\widehat{f}_{N,r}(\bm{k})|^2\le\frac{|\widehat{f}(\bm{k})|^2}{2}   \right\} \right]\nonumber\\
        &\le \mathbb{E}_{\bm{D}_r}\left[\bm{1}\left\{|\widehat{f}_{N,r}(\bm{k})|^2\le \mathbb{E}_{\bm{D}_r}\left[|\widehat{f}_{N,r}(\bm{k})|^2\right]-\frac{|\widehat{f}(\bm{k})|^2}{2}   \right\} \right]\nonumber\\
        &\le \frac{\sigma^2(\bm{k})}{\sigma^2(\bm{k})+|\widehat{f}_{N,r}(\bm{k})|^4/4}=\frac{1}{1+|\widehat{f}_{N,r}(\bm{k})|^4/(4\sigma^2(\bm{k}))}.\label{prob_m_hat}
    \end{align}
    When $S(\bm{k})\le \frac{\Lambda_2}{\delta N}$ and $t(\bm{k})\le\frac{\Lambda_1}{\delta N} $, Lemma \ref{lemma_sigma} gives 
    \begin{equation*}
        \frac{\sigma^2(\bm{k})}{|\widehat{f}(\bm{k})|^4}\le \frac{8\Lambda_2}{\delta N|\widehat{f}(\bm{k})|^2}+\frac{2\Lambda_1^2\Lambda_2}{(\delta N)^3|\widehat{f}(\bm{k})|^4}.
    \end{equation*}
    Applying Lemmas \ref{lemma_lambda} and \ref{lemma_psi} with $\theta \in (0,1)$ and $\eta  = (1-1/(2\alpha+2\lambda))\theta$, combined with the definition of $G_{\alpha,\lambda,\beta,\bm{\gamma},\delta,\theta}$, we obtain
    \begin{align*}
        &\frac{8\Lambda_2}{\delta N|\widehat{f}(\bm{k})|^2}+\frac{2\Lambda_1^2\Lambda_2}{(\delta N)^3|\widehat{f}(\bm{k})|^4}\\
        &\le \frac{8\Upsilon_{\alpha,\lambda,\beta, \bm{\gamma},\delta,\theta} \Gamma_{\alpha,\lambda,\beta,\bm{\gamma}}(m)^{2\theta(\alpha+\lambda)}}{\delta N^{2\theta(\alpha+\lambda)+1}|\widehat{f}(\bm{k})|^2} + \frac{2\Upsilon_{\alpha,\lambda,\beta, \bm{\gamma},\delta,\theta}\Theta_{\alpha,\lambda,\beta,\bm{\gamma},\delta,\eta}^2 \Gamma_{\alpha,\lambda,\beta,\bm{\gamma}}(m)^{\theta(4\alpha+4\lambda-1)}}{\delta^3 N^{\theta(4\alpha+4\lambda-1)+3} |\widehat{f}(\bm{k})|^4}\\
        &\le \frac{8\Upsilon_{\alpha,\lambda,\beta, \bm{\gamma},\delta,\theta} }{\delta N^{1-\theta} G_{\alpha,\lambda,\beta,\bm{\gamma},\delta,\theta}^2}+\frac{2\Upsilon_{\alpha,\lambda,\beta, \bm{\gamma},\delta,\theta}\Theta_{\alpha,\lambda,\beta,\bm{\gamma},\delta,\eta}^2}{\delta^3 N^{3-3\theta}G_{\alpha,\lambda,\beta,\bm{\gamma},\delta,\theta}^4}\\
        &\le \frac{\delta}{8(1-\delta)^{\frac{1}{2}}}+\frac{\Theta_{\alpha,\lambda,\beta,\bm{\gamma},\delta,\eta}^2}{32\delta(1-\delta)^{\frac{1}{2}}G_{\alpha,\lambda,\beta,\bm{\gamma},\delta,\theta}^2}\\
        &\le \frac{\delta}{8(1-\delta)} + \frac{\delta}{8(1-\delta)}=\frac{\delta}{4(1-\delta)}.
    \end{align*}
    where the second inequality used $ \Gamma_{\alpha,\lambda,\beta,\bm{\gamma}}(m)\ge 1$ and the third inequality used $N\ge 1$. So the right-most side of (\ref{prob_m_hat}) is bounded by $\delta$. Hence, by Lemma \ref{lemma_s_r}, 
    \begin{align*}
        &\operatorname{Pr}\left(|\widehat{f}_{N,r}(\bm{k})|^2\le\frac{|\widehat{f}(\bm{k})|^2}{2} \right)\\
        &\le \operatorname{Pr}\left(\left\{S(\bm{k})>\frac{\Lambda_2}{\delta N} \right\}\cup\left\{t(\bm{k})>\frac{\Lambda_1}{\delta N} \right\} \right)\\
        &\quad + \operatorname{Pr}\left(\left\{|\widehat{f}_{N,r}(\bm{k})|^2\le\frac{|\widehat{f}(\bm{k})|^2}{2} \right\}\cap\left\{S(\bm{k})\le\frac{\Lambda_2}{\delta N} \right\}\cap\left\{t(\bm{k})\le\frac{\Lambda_1}{\delta N} \right\}  \right)\\
        &\le 5\delta.
    \end{align*}
    The conclusion then follows from Lemma \ref{lemma_median_trick}.
\end{proof}

\begin{theorem}\label{thm_term3}
    Let $\delta\in(0,\frac{1}{10}),N=2^m, \theta\in (0,1)$, and let $R$ be an odd natural number satisfying
    \begin{equation*}
        \delta_2:= \frac{|K|}{2}(8\delta(1-2\delta))^{R/2} + \frac{N}{2}(20\delta(1-5\delta))^{R/2}<1.
    \end{equation*}
    Assume $K\supseteq K_{m,\bm{\gamma}}$ and the condition (\ref{condition_E}) holds.
    Then, for any $f$ with $\Vert f\Vert_{s,\alpha,\lambda}<\infty$, with probability at least $1-\delta_2$ we have
    \begin{align*}
        \sum_{\bm{k}'\in H_N'\setminus H^{E,\epsilon}_N}|\widehat{f}(\bm{k}')|^2&\le G_{\alpha,\lambda,\beta,\bm{\gamma},\delta,\theta}^2\frac{ \Gamma_{\alpha,\lambda,\beta,\bm{\gamma}}(m)^{2\theta(\alpha+\lambda)}  }{N^{2\theta(\alpha+\lambda)+\theta-1} }\Vert f\Vert_{s,\alpha,\lambda}^2+4N\epsilon_1\left(2\Vert f\Vert_{\infty}+\epsilon_1\right) \\
        &\quad+\frac{2(2-\delta)}{\delta(1-\delta)}\Upsilon_{\alpha,\lambda,\beta, \bm{\gamma},\delta,\theta}\frac{\Gamma_{\alpha,\lambda,\beta,\bm{\gamma}}(m)^{2\theta(\alpha+\lambda)}}{N^{2\theta(\alpha+\lambda)}}\Vert f\Vert_{s,\alpha,\lambda}^2,
    \end{align*}
    where $\epsilon_1$, $\Gamma_{\alpha,\lambda,\beta,\bm{\gamma}}(m)$, $\Upsilon_{\alpha,\lambda,\beta, \bm{\gamma},\delta,\theta}$, and $G_{\alpha,\lambda,\beta,\bm{\gamma},\delta,\theta}$ are defined in (\ref{constant_ep1}),  (\ref{Gamma(m)}), (\ref{constant_Gamma}), and (\ref{constant_G}), respectively.
\end{theorem}
\begin{proof}
    Applying Corollary \ref{cor_K_minus_H} to all $\bm{k}\in K\setminus H_N'$ and Lemma \ref{lemma_with_condition} to all $\bm{k}'\in H_N'$ satisfying (\ref{condition_hat_f}), a union bound argument shows that, with probability at least $1-\delta_2$, we have:
    \begin{itemize}
        \item For all $\bm{k}\in K\setminus H_N'$,
        \begin{equation*}
            \underset{r\in 1:R}{\operatorname{med}}|\widehat{f}_{N,r}(\bm{k})|^2\le\frac{(2-\delta)\Lambda_2}{\delta(1-\delta)N}.
        \end{equation*}
        \item For all $\bm{k}'\in H_N'$ satisfying (\ref{condition_hat_f}),
        \begin{equation*}
            \underset{r\in 1:R}{\operatorname{med}}|\widehat{f}_{N,r}(\bm{k}')|^2> \frac{|\widehat{f}(\bm{k}')|^2}{2}.
        \end{equation*}
    \end{itemize}
    Now, for any $\bm{k}'\in H_N'\setminus H^{E,\epsilon}_N$, either
    \begin{equation*}
        |\widehat{f}(\bm{k}')|< G_{\alpha,\lambda,\beta,\bm{\gamma},\delta,\theta} \frac{\Gamma_{\alpha,\lambda,\beta,\bm{\gamma}}(m)^{\theta(\alpha+\lambda)}}{N^{\theta(\alpha+\lambda+1/2)}}\Vert f\Vert_{s,\alpha,\lambda},
    \end{equation*}
    or, otherwise, we have $|\widehat{f}(\bm{k}')|^2/2<\underset{r\in 1:R}{\operatorname{med}}|\widehat{f}_{N,r}(\bm{k}')|^2$. Under the latter condition, for any $\bm{k}\in H^{E,\epsilon}_N\setminus H_N'$, by the definition of $H^{E,\epsilon}_N$,
    \begin{align*}
        \underset{r\in 1:R}{\operatorname{med}}|\widehat{f}^{E,\epsilon}_{N,r}(\bm{k}')|^2\le \underset{r\in 1:R}{\operatorname{med}}|\widehat{f}^{E,\epsilon}_{N,r}(\bm{k})|^2.
    \end{align*}
    Under condition \eqref{condition_E}, we apply Theorem \ref{thm_dis} with $\widetilde{\bm{k}} \in \{\bm{k}, \bm{k}'\}$ to obtain:
    \begin{equation*}
        \left||\widehat{f}^{E,\epsilon}_{N,r}(\widetilde{\bm{k}})|- |\widehat{f}_{N,r}(\widetilde{\bm{k}})|\right|\le \epsilon_1.
    \end{equation*}
    So
    \begin{align}
            \left||\widehat{f}^{E,\epsilon}_{N,r}(\widetilde{\bm{k}})|^2- |\widehat{f}_{N,r}(\widetilde{\bm{k}})|^2\right|
            &=\left||\widehat{f}^{E,\epsilon}_{N,r}(\widetilde{\bm{k}})|-|\widehat{f}_{N,r}(\widetilde{\bm{k}})|\right|\left(|\widehat{f}^{E,\epsilon}_{N,r}(\widetilde{\bm{k}})|+|\widehat{f}_{N,r}(\widetilde{\bm{k}})|\right)\nonumber\\
            &\le \epsilon_1\left(2|\widehat{f}_{N,r}(\widetilde{\bm{k}})|+\epsilon_1\right) \le \epsilon_1\left(2\Vert f\Vert_{\infty}+\epsilon_1\right).\nonumber
    \end{align}
    According to Lemma \ref{lemma_median_perserve}, we have
    \begin{align*}
        \left|\underset{r\in 1:R}{\operatorname{med}}|\widehat{f}_{N,r}(\widetilde{\bm{k}})|^2-\underset{r\in 1:R}{\operatorname{med}}|\widehat{f}^{E,\epsilon}_{N,r}(\widetilde{\bm{k}})|^2\right| \le\epsilon_1\left(2\Vert f\Vert_{\infty}+\epsilon_1\right).
    \end{align*}
    Hence, we have 
    \begin{align*}
        \frac{|\widehat{f}(\bm{k}')|^2}{2}&<\underset{r\in 1:R}{\operatorname{med}}|\widehat{f}_{N,r}(\bm{k}')|^2\le\underset{r\in 1:R}{\operatorname{med}}|\widehat{f}^{E,\epsilon}_{N,r}(\bm{k}')|^2 + \epsilon_1\left(2\Vert f\Vert_{\infty}+\epsilon_1\right)\\
        &\le \underset{r\in 1:R}{\operatorname{med}}|\widehat{f}^{E,\epsilon}_{N,r}(\bm{k})|^2+\epsilon_1\left(2\Vert f\Vert_{\infty}+\epsilon_1\right)\\
        &\le \underset{r\in 1:R}{\operatorname{med}}|\widehat{f}_{N,r}(\bm{k})|^2+2\epsilon_1\left(2\Vert f\Vert_{\infty}+\epsilon_1\right)\\
        &\le \frac{(2-\delta)\Lambda_2}{\delta(1-\delta)N}+2\epsilon_1\left(2\Vert f\Vert_{\infty}+\epsilon_1\right).
    \end{align*}
    Therefore, 
    \begin{align*}
        &\sum_{\bm{k}'\in H_N'\setminus H^{E,\epsilon}_N}|\widehat{f}(\bm{k}')|^2\\
        &\le  N G_{\alpha,\lambda,\beta,\bm{\gamma},\delta,\theta}^2 \frac{\Gamma_{\alpha,\lambda,\beta,\bm{\gamma}}(m)^{2\theta(\alpha+\lambda)}}{N^{2\theta(\alpha+\lambda+1/2)}}\Vert f\Vert_{s,\alpha,\lambda}^2
        + N \left(\frac{2(2-\delta)\Lambda_2}{\delta(1-\delta)N}+4\epsilon_1\left(2\Vert f\Vert_{\infty}+\epsilon_1\right)\right)\\
        &\le G_{\alpha,\lambda,\beta,\bm{\gamma},\delta,\theta}^2\frac{ \Gamma_{\alpha,\lambda,\beta,\bm{\gamma}}(m)^{2\theta(\alpha+\lambda)}  }{N^{2\theta(\alpha+\lambda)+\theta-1} } \Vert f\Vert_{s,\alpha,\lambda}^2
        +4N\epsilon_1\left(2\Vert f\Vert_{\infty}+\epsilon_1\right)\\
        &\quad+\frac{2(2-\delta)}{\delta(1-\delta)}\Upsilon_{\alpha,\lambda,\beta, \bm{\gamma},\delta,\theta}\frac{\Gamma_{\alpha,\lambda,\beta,\bm{\gamma}}(m)^{2\theta(\alpha+\lambda)}}{N^{2\theta(\alpha+\lambda)}}\Vert f\Vert_{s,\alpha,\lambda}^2,
    \end{align*}
    where the last inequality used Lemma \ref{lemma_lambda}.
\end{proof}

Finally, by substituting the bounds from Theorems \ref{thm_term1}, \ref{thm_term2}, and \ref{thm_term3} back into the decomposition \eqref{decomp}, and applying the union bound, we have that with probability at least $1-\delta_1-\delta_2\ge 1-\delta'$,
\begin{align*}
    \Vert A^{E,\epsilon}_m(f)-f \Vert_{L^2}^2&\le \frac{6-3\delta-\delta^2}{\delta(1-\delta)}
        \Gamma_{\alpha,\lambda,\beta,\bm{\gamma},\delta,\theta}\frac{\Gamma_{\alpha,\lambda,\beta,\bm{\gamma}}(m)^{2\theta(\alpha+\lambda)}}{N^{2\theta(\alpha+\lambda)}}\Vert f\Vert_{s,\alpha,\lambda}^2+6N\epsilon_1^2\\
        &\quad+G_{\alpha,\lambda,\beta,\bm{\gamma},\delta,\theta}^2\frac{ \Gamma_{\alpha,\lambda,\beta,\bm{\gamma}}(m)^{2\theta(\alpha+\lambda)}  }{N^{2\theta(\alpha+\lambda)+\theta-1} }\Vert f\Vert_{s,\alpha,\lambda}^2+8N\epsilon_1 \Vert f\Vert_{\infty}\\
        &\le 
        C_{\alpha,\lambda,\beta,\bm{\gamma},\delta,\theta}\frac{ \Gamma_{\alpha,\lambda,\beta,\bm{\gamma}}(m)^{2\theta(\alpha+\lambda)}  }{N^{2\theta(\alpha+\lambda)+\theta-1} }\Vert f\Vert_{s,\alpha,\lambda}^2+8N\epsilon_1(\Vert f\Vert_{\infty}+\epsilon_1),
\end{align*}
which concludes the proof of Theorem \ref{thm_main}.

\section{Proof of Theorem~\ref{thm_card_univ}}\label{appendix_thm_card_univ}

Our proof relies on the following two lemmas.

\begin{lemma}\label{lemma_K^alpha}
There exists a constant $C\ge 1$ such that for all $\alpha\in \mathbb{N}$ and $\tau\ge 0$,
$$|K^{\operatorname{univ}}_{\alpha}(\tau)|\le 2^m\prod_{j=1}^s\left(1+Cj^{-\tau}(1+m)\right).$$
\end{lemma}

\begin{proof}
   It suffices to prove $A_{\alpha-1,1}\le B_{\alpha-1,1}\le C B_{\alpha-1,0}$ for $\alpha\ge 2$, since the remaining argument is analogous to that used in Theorem~\ref{thm_card_alpha}. 

   For $\alpha\ge 2$, let $x = (2^{1/(\alpha-1)}-1)^{-1}$, $y = (2^{1/\alpha}-1)^{-1}$, and $\eta = y/x>1$. Then
    \begin{equation*}
        \frac{B_{\alpha-1,1}}{B_{\alpha-1,0}}=\frac{\sum_{k=1}^{\alpha-1}y^k/k! }{\sum_{k=1}^{\alpha-1}x^k/k!}\le \eta^\alpha\le \frac{\sum_{k=\alpha}^{\infty}y^k/k! }{\sum_{k=\alpha}^{\infty}x^k/k!}.
    \end{equation*}
     Thus,
    \begin{equation*}
        \frac{B_{\alpha-1,1}}{B_{\alpha-1,0}}\le \frac{B_{\alpha-1,1}+ \sum_{k=\alpha}^{\infty}y^k/k!}{B_{\alpha-1,0}+\sum_{k=\alpha}^{\infty}x^k/k! }=\frac{e^y-1}{e^x-1}
    \end{equation*}
   Taking the $\limsup$ as $\alpha\to\infty$ yields
   \begin{align*}
       \limsup_{\alpha\to\infty} \frac{B_{\alpha-1,1}}{B_{\alpha-1,0}}&\le \lim_{\alpha\to\infty}\frac{e^{(2^{1/\alpha}-1)^{-1}} -1}{e^{(2^{1/(\alpha-1)}-1)^{-1}} -1}\\
       &=\lim_{\alpha\to\infty} \exp\left(\frac{1}{2^{1/\alpha}-1}-\frac{1}{2^{1/(\alpha-1)}-1}\right) =e^{1/\log 2}.
   \end{align*}
   Hence, $B_{\alpha-1,1}\le C B_{\alpha-1,0}$ for some constant $C\ge 1$. The bound $A_{\alpha-1,1}\le B_{\alpha-1,1}$ follows directly from (\ref{constant_A_B}).
\end{proof}

\begin{lemma}\label{K^alpha}
Given $\tau\geq 0$, if the weights $\gamma_u$ satisfy (\ref{restrict_gamma}) for some $\alpha\in\mathbb{N}_0$ and $\lambda\in(0,1]$, then  
    $$K_{m,\alpha,\lambda,\beta,\bm{\gamma},\delta}\subseteq  K^{\operatorname{univ}}_{\alpha}(\tau)\bigcup K^{\operatorname{univ}}_{\alpha+1}(\tau)$$
    with $\delta\in (0,1/2)$ and $\beta=\Gamma^{\frac{1}{\alpha+\lambda+1/2}}\max(1,B_{\alpha,\lambda})$.
\end{lemma}
\begin{proof}
    For $\alpha=0$ and any $\lambda\in (0,1]$, we have $B_{0,\lambda}=0$ in (\ref{constant_A_B}). According to (\ref{T_u_m}) and the fact that $\mu_{0,\lambda}(k)=\lambda\mu_{0,1}(k)$ for $k\in\mathbb{N}$, we have 
    \begin{align*}
    K_{u,0,\lambda}(T_{u,m,0,\lambda,\beta,\bm{\gamma},\delta})& \subseteq K_{u,0,\lambda}(T^{\operatorname{univ}}_{u,m,0,\lambda}(\tau))\\
    &=\left\{\bm{k}\in \mathbb{N}_0^s:\bm{s}(\bm{k})=u,\mu_{0,\lambda}(\bsk)\le T^{\operatorname{univ}}_{u,m,0,\lambda}(\tau) \right\}\\
    &=\left\{\bm{k}\in \mathbb{N}_0^s:\bm{s}(\bm{k})=u,\mu_{0,1}(\bsk)\le T^{\operatorname{univ}}_{u,m,1}(\tau) \right\}\\
    &=K_{u,0,1}\left(T^{\operatorname{univ}}_{u,m,1}(\tau)\right).
\end{align*}
It follows that $K_{m,0,\lambda,\beta,\bm{\gamma},\delta}\subseteq  K^{\operatorname{univ}}_{0}(\tau)$.

Now we consider $\alpha\in\mathbb{N}$. 
According to (\ref{constant_A_B}), we have $B_{\alpha-1,1}\leq B_{\alpha,0}\leq B_{\alpha,\lambda}$ for all $\alpha\in\mathbb{N}$ and $\lambda\in (0,1]$. Therefore,
 \begin{align*}
\frac{T^{\operatorname{univ}}_{u,m,\alpha,\lambda}(\tau)}{\alpha+\lambda}=&m -\log_2\left(1+B_{\alpha,\lambda}^{|u|} \right)-\tau\sum_{j\in u}\log_2(j) \\
\le &m -\log_2\left(1+ B_{\alpha,0}^{|u|} \right)-\tau\sum_{j\in u}\log_2(j)=\frac{T^{\operatorname{univ}}_{u,m,\alpha+1}(\tau)}{\alpha+1}\\
\le &m -\log_2\left(1+ B_{\alpha-1,0}^{|u|} \right)-\tau\sum_{j\in u}\log_2(j)=\frac{T^{\operatorname{univ}}_{u,m,\alpha}(\tau)}{\alpha}.
 \end{align*}
For every $\bsk\in K_{u,\alpha,\lambda}(T_{u,m,\alpha,\lambda,\beta,\bm{\gamma},\delta})$, we use $\lambda(\alpha+1)+(1-\lambda)\alpha=\alpha+\lambda$ and derive
\begin{align*}
 \mu_{\alpha,\lambda}(\bsk)&=\lambda\mu_{\alpha,1}(\bsk)+(1-\lambda)\mu_{\alpha-1,1}(\bsk)\leq T_{u,m,\alpha,\lambda,\beta,\bm{\gamma},\delta}\leq T^{\operatorname{univ}}_{u,m,\alpha,\lambda}(\tau)\\
 &\leq \lambda T^{\operatorname{univ}}_{u,m,\alpha+1}(\tau)+(1-\lambda)T^{\operatorname{univ}}_{u,m,\alpha}(\tau).    
\end{align*}
Hence, either $\mu_{\alpha,1}(\bsk)\leq T^{\operatorname{univ}}_{u,m,\alpha+1}(\tau)$ or $\mu_{\alpha-1,1}(\bsk)\leq T^{\operatorname{univ}}_{u,m,\alpha}(\tau)$, which implies our conclusion.
\end{proof}

\begin{proof}[Proof of Theorem~\ref{thm_card_univ}]
     The second assertion follows directly from Lemma \ref{K^alpha}. To prove the first assertion, we show that for any nonempty $u\subseteq 1{:}s$, there are only finitely many $\alpha\in\mathbb{N}_0$ such that $T^{\operatorname{univ}}_{u,m,\alpha}(\tau)\ge 0$. 
    
    For sufficiently large $\alpha$, the inequality $2^{1/\alpha}-1\le 2/\alpha$ yields
    \begin{equation*}
        B_{\alpha,0}=\sum_{k=1}^{\alpha}\frac{1}{k!(2^{1/\alpha}-1)^k}\ge \sum_{k=1}^{\alpha}\frac{1}{k!}\left(\frac{\alpha}{2}\right)^k\ge \sum_{k=1}^{\lfloor \alpha/2\rfloor }\frac{1}{k!}(\lfloor \alpha/2\rfloor)^k.
    \end{equation*}
    It follows that $\liminf_{\alpha\to\infty}e^{-\lfloor \alpha/2\rfloor}B_{\alpha,0}\ge 1/2$ due to the limit
    $$\lim_{n\to\infty}e^{-n}\sum_{k=0}^{n}\frac{n^k}{k!}=\frac{1}{2}.$$
    (For a one-line proof, apply the Central Limit Theorem to Poisson distribution with mean $n$). 
    Consequently, there exists $\alpha_0>0$ such that for all $\alpha>\alpha_0$, 
    \begin{equation*}
        B_{\alpha,0}\ge \frac{e^{\lfloor \alpha/2\rfloor}}{4}\ge 1.
    \end{equation*}
    For such $\alpha$ and $u\ne \emptyset$, it follows that
    \begin{align*}
        \frac{T^{\operatorname{univ}}_{u,m,\alpha+1}(\tau)}{\alpha+1}\le m-\log_2\left(1+B_{\alpha,0}^{|u|} \right)
        \le m-\log_2(B_{\alpha,0}) \le m+2-\lfloor \alpha/2\rfloor\log_2 e.
    \end{align*}
    Therefore, for $\alpha\geq \max\{\alpha_0,2m+4\}$, we have $m+2-\lfloor \alpha/2\rfloor\log_2 e< 0$, and $T^{\operatorname{univ}}_{u,m,\alpha+1}(\tau)<0$ for all $u\neq \emptyset$ and $\tau\ge 0$.
    
    Hence, $K^{\operatorname{univ}}_{\alpha+1}(\tau)=\{\bm{0}\}$ for all $\tau\ge 0$ whenever $\alpha\ge (\alpha_0+6) m $. Applying Lemma \ref{lemma_K^alpha} yields the bound
    \begin{equation*}
        |K^{\operatorname{univ}}(\tau)|\le \sum_{\alpha=1}^{\lfloor (\alpha_0+6) m \rfloor}|K^{\operatorname{univ}}_\alpha(\tau)|\le (\alpha_0+6) m2^m\prod_{j=1}^s\left(1+Cj^{-\tau}(1+m)\right). 
    \end{equation*}
    The proof is complete after taking $C^{\operatorname{univ}}=\max(\alpha_0+6,C)$.
\end{proof}

\section{Proof of Proposition \ref{prop_gray_code}}\label{appendix_gray_code}
\begin{proof}
For the first assertion, let $R_j^i = \mu_{\alpha-1,1}(k_j^i)$ be the associated compositions. We consider the following two cases. If $(R_1^1, \dots, R_{|u|}^1) = (R_1^2, \dots, R_{|u|}^2)$, let $j$ be the first index such that $\bm{w}_j^1 \neq \bm{w}_j^2$. For each $i > j$, $\bm{w}_i^2$ is the first tuple in the ordering, namely $(R_i^1, 0, \dots, 0)$, while $\bm{w}_i^1$ is the last: either $(R_i^1-1, 1, 0, \dots, 0)$ if \eqref{R_w_eq} admits multiple solutions, or $(R_i^1, 0, \dots, 0)$ otherwise. In this case, $\bm{w}_i^1$ and $\bm{w}_i^2$ differ in at most 2 positions for $i > j$, while $\bm{w}_j^1$ and $ \bm{w}_j^2$ differ in at most $\alpha$ positions. Thus, the sequences $(\bm{w}_1^1, \dots, \bm{w}_{|u|}^1)$ and $(\bm{w}_1^2, \dots, \bm{w}_{|u|}^2)$ differ in at most $2|u| + \alpha - 2$ positions.

If, instead, $(R_1^2, \dots, R_{|u|}^2)$ is the successor of $(R_1^1, \dots, R_{|u|}^1)$, then for each $1 \le i \le |u|$, $\bm{w}_i^2$ and $\bm{w}_i^1$ are again the first and last tuples in their respective orderings. Specifically, $\bm{w}_i^2 = (R_i^2, 0, \dots, 0)$, whereas $\bm{w}_i^1$ is either $(R_i^1-1, 1, 0, \dots, 0)$ or $(R_i^1, 0, \dots, 0)$. Thus, the sequences $(\bm{w}_1^1, \dots, \bm{w}_{|u|}^1)$ and $(\bm{w}_1^2, \dots, \bm{w}_{|u|}^2)$ differ in at most $2|u|\le 2|u|+\alpha-1$ positions.

For the second assertion, the argument proceeds similarly. If $(\bm{w}_1^1, \dots, \bm{w}_{|u|}^1) = (\bm{w}_1^2, \dots, \bm{w}_{|u|}^2)$, let $j$ be the first index such that $\bm{v}_j^1 \neq \bm{v}_j^2$. Since $\bm{v}_j^2$ is the successor of $\bm{v}_j^1$ in the Gray code ordering, they differ by exactly one bit. For $i > j$, $\bm{v}_i^2$ is the initial tuple in the ordering, which is the zero vector $\bm{0}$, while $\bm{v}_i^1$ is the terminal tuple: it is $\bm{0}$ if $z_i < \alpha$, and $(0, \dots, 0, 1)$ of length $w_{i, z_i} - 1$ if $z_i = \alpha$. In either case, $\bm{v}_i^1$ and $\bm{v}_i^2$ differ in at most one position.

If, instead, $(\bm{w}_1^2, \dots, \bm{w}_{|u|}^2)$ is the successor of $(\bm{w}_1^1, \dots, \bm{w}_{|u|}^1)$, then for each $1 \le i \le |u|$, $\bm{v}_i^1$ is the terminal tuple and $\bm{v}_i^2$ is the initial tuple of their respective orderings. Specifically, $\bm{v}_i^2$ is the zero vector $\bm{0}$, while $\bm{v}_i^1$ is either $\bm{0}$ or $(0, \dots, 0, 1)$. Although the lengths of $\bm{v}_i^1$ and $\bm{v}_i^2$ may differ (due to changes in $w_{i, z_i}$), they still differ in at most one bit under zero-padding. 
\end{proof}

\bibliographystyle{abbrv} 
\bibliography{ref}
\end{document}